\documentclass[MSNbibl,number,citesort,dvips]{arxbj}

\aid{0}
\volume{20}
\issue{2}
\pubyear{2014}
\firstpage{878}
\lastpage{918}
\doi{10.3150/13-BEJ509} %kopijuoti is PTS

\makeatletter
\def\sc{\uppercase}

\newcommand{\rrvert}{\vert}
\newcommand{\llvert}{\vert}
\renewcommand{\epsilon}{\varepsilon}
\newcommand{\eqref}[1]{(\ref{#1})}
\newtheorem{theo}{Theorem}
\newtheorem{hyp}{Hypothesis}
\newtheorem{lemme}[theo]{Lemma}
\newtheorem{prop}[theo]{Proposition}
\newproclaim{defi}{Definition}
\newproclaim{deffi}{Definition}

\newremark{remi}{Remark}
\def\A{\mathcal{A}}
\def\R{\mathbf{R}}
\def\N{\mathbf{N}}
\def\E{\mathbf{E}}
\def\P{\mathbf{P}}
\def\d{\mathrm{d}}
\def\v{\mathrm{v}}
\def\1{\mathbf{1}}
\def\t{\theta}
\def\h{\mathfrak{h}}
\def\hh{\mathfrak{h}}
\def\mm{\mathfrak{m}}
\def\ppp{\mathfrak{p}}
\def\bb{\mathfrak{b}}
\def\eps{\varepsilon}
\def\ffi{\varphi}
\makeatother

\begin{document}
\begin{frontmatter}

\title{Nonparametric inference for fractional diffusion}
\runtitle{Nonparametric inference for fractional diffusion}

\begin{aug}
%%%% inicialai - be tarpu
\author{\fnms{Bruno} \snm{Saussereau}\corref{}\ead[label=e1]{bruno.saussereau@univ-fcomte.fr}}% \and
\runauthor{B. Saussereau} %% auto
\address{Universit\'e de Franche-Comt\'e,
Laboratoire de Math\'{e}matiques de Besan\c{c}on, UMR CNRS 6623, 16
route de Gray, 25030 Besan\c{c}on, France. \printead{e1}}
\end{aug}

% HISTORY:
\received{\smonth{5} \syear{2012}}
\revised{\smonth{12} \syear{2012}}

% ABSTRACT
%
\begin{abstract}
A non-parametric diffusion model with an additive fractional Brownian
motion noise is considered in this work.
The drift is a non-parametric function that will be estimated by two
methods. On one hand, we propose a locally linear estimator based on
the local approximation of the drift by a linear function. On the other
hand, a Nadaraya--Watson kernel type estimator is studied. In both
cases, some non-asymptotic results are proposed by means of deviation
probability bound. The consistency property of the estimators are
obtained under a one sided dissipative Lipschitz condition on the drift
that insures the ergodic property for the stochastic differential
equation. Our estimators are first constructed under continuous
observations. The drift function is then estimated with discrete time
observations that is of the most importance for practical applications.
\end{abstract}

% KEYWORDS
% visi is mazosios raides ir pagal abecele
\end{frontmatter}

%s1 #&#
\section{Introduction}\label{sec1}
The inference problem for diffusion process is now a well understood
problem. The inference based on discretely observed diffusion is very
important from a practical point of view and it has also benefited from
numerous studies.
With the development of technology, differential equations driven by
noise with memory is increasingly popular in the statistical community
as a modelling device. The subject of this work concerns the
non-parametric estimation problem of the drift coefficient of a
fractional diffusion
described by the scalar equation
%
%e1 #&#
\begin{equation}
\label{eds} X_t = x_0 + \int_{0}^t
b(X_s) \,\d s + B_t^H ,\qquad t\ge0,
\end{equation}
where $x_0\in\mathbf R$ is the initial value of the process
$X=(X_t)_{t\ge0}$, and $B^H=(B_t^H)_{t\ge0}$ is
a fractional Brownian motion (fBm in short) with Hurst parameter $H\in
(0,1)$. This means that $B^H$ is a
Gaussian process, centered, starting from $0$ and such that $\mathbf E
(B^H_t-B^H_s)^2 = |t-s|^{2H}$. Therefore,
the process $B^H$ has $\hh$-H\"{o}lder continuous paths for all $\hh
\in(0,H)$.
If $H=1/2$, then $B^{H}$ is clearly a Brownian motion and we refer to
\cite{n}, Chapter 5, for a survey about the fBm.

Stochastic differential equations driven by fBm have recently carried
out a lot of development. The special case
of a constant diffusion coefficient is more specifically treated in
\cite{no} where it is proved that equation \eqref{eds} has a strong
unique solution
if we assume the linear growth condition $|b(x)|\le c_b(1+|x|)$ for $b$
when $H<1/2$, and H\"{o}lder continuity of order
$\bb\in(1-1/2H , 1)$ when $H>1/2$.
In this paper, we assume that these conditions are true.

If we suppose that the observation process is
\[
X_t = x_0 + \int_{0}^t
b(X_s) \,\d s + \sigma B_t^H ,\qquad t\ge0
\]
with some unknown diffusion coefficient $\sigma$, then one may
estimate the unknown Hurst parameter~$H$ and the diffusion coefficient
$\sigma$ via the quadratic variation (see \cite{begyn,coeur,il}).
This is the reason why we restrict ourselves to the case $\sigma=1$
and $H$ known.

Almost all the existing articles relate to the parametric case when one
consider the model
%
%e2 #&#
\begin{equation}
\label{model2} X_t = x_0 + \int_{0}^t
\theta b(X_s) \,\d s + B_t^H , \qquad t\ge0 ,
\end{equation}
with $\theta$ is the unknown parameter. Let us briefly review the
works that have been already done.
When the drift is linear, $X$ is the fractional Ornstein--Uhlenbeck
process and the estimation of~$\theta$ has attracted a lot of
attention. This problem has been first tackled by \cite{kl} using a
maximum likelihood procedure. Some least square estimates are proposed
in \cite{hn-stat,beo}. See also \cite{rao} for other methods.
When $b$ is not necessary linear, the pioneering work is \cite{tv}
(see also the extended electronic version \cite{tv-hal}). In this
paper, the maximum likelihood estimator of~$\theta$ is studied both
with continuous and discrete observations. One may also refer to \cite{bishwal}, Chapter 6 and \cite{mishura}, Chapter~6. A moment matching
estimation is done in \cite{pl} and let us finally mention that a
general discrete data maximum likelihood is proposed in \cite{ct} and
a least square method is studied in \cite{nt} for the parametric
estimation problem for model described by the equation \eqref{model2}.

To our knowledge, there is only one paper dealing with non-parametric
estimation. In \cite{mp}, the
authors consider the model
%
%e3 #&#
\begin{equation}
\label{trend} X_t = x_0 + \int_{0}^t
b(X_s) \,\d s + \epsilon B_t^H ,\qquad t\ge0 ,
\end{equation}
and proposed a kernel type estimator of the trend coefficient
$b_s:=b(x_s)$ where $(x_s)_{s\ge0}$ is the solution of equation \eqref
{trend} when $\epsilon= 0$. The asymptotic behaviour is discussed when
$\epsilon\to0$ on a finite time horizon when $H>1/2$.

Our problem is of different nature than the previous ones. We will
investigate two procedures to estimate the unknown function $b$ at a
fixed point $x \in\mathbf R$, that is, $b(x)$. We start with
estimators based on continuous observation of $X$. Then, using a
discretization, we propose estimators based on discrete time data which
are the most important for practical applications. It is difficult to
work directly with the fBm $B^H$ because it is not a semimartingale. So
we use the fundamental martingale (so called in \cite{kl}, see also
\cite{nvv}) that will have nicer properties.
Then some simple and classical ideas lead us to the construction of two
estimators of $b(x)$.
First, we consider that the drift is a constant function in a
neighbourhood of the point $x$ and thus the problem becomes parametric.
The form of the least square estimator in this parametric case is used
to propose a kernel type estimator of Nadaraya--Watson type (see \eqref
{defnw} in Definition \ref{def1}). If the drift is assumed to be
linear in a small vicinity of $x$, then by similar arguments we define
a locally linear estimator of the drift function in the point $x$ (see
\eqref{Yhatb} in Definition \ref{def2}). These local linear smoothers
are known to avoid some undesirable edge effects.

In order to study our these two estimators, we apply the same strategy.
We prove some deviation probability bounds using non-asymptotic
approach. Our probability bounds are stated conditionally to a random set.
This method has been employed in \cite{spo} for a diffusion process
with drift and variance given as non-parametric functions of the state
variable. So a first step will consist to mainly focus ourselves on a
non-asymptotic approach for which there is no difference between the
ergodic and non-ergodic cases. Similar probability bounds are valid for
the discretization of our estimators.

Then we investigate the consistency. With kernel type estimators, it is
clear that it is necessary to impose some conditions which provide that
the observed process $(X_t)_{t\ge0}$ returns to any vicinity of the
point $x$ infinitely many times. The ergodicity can guarantee this
property in the classical Brownian case (see \cite{kou}). The null
recurrence of $X$ can also be invoked when $H=1/2$ as in \cite
{loulou}. We refer to \cite{ll} for the case of Harris recurrent
diffusion. In the fractional case such ergodic properties, will hold
under the assumption that the drift has polynomial growth and satisfies
a one-sided dissipative Lipschitz condition (see \cite{kn}, e.g.).
Starting from our conditional deviation probability bound, we shall
prove that the probability of the random event with respect to which
the results are stated converges to 1 under the cited above hypotheses
on the drift $b$.
Thus, the weak consistency (this mean that the convergence holds in
probability) is proved for both estimators, for continuous and discrete
observations.

The most important results of our paper are certainly the ones
concerning the problem of the estimation of the unknown value of the
drift $b$ in a fixed point $x$ under discrete observations of the
process $X$. For simplicity, we describe the Nadaraya--Watson estimator.
For equally spaced observation times $\{t_k\}_{0\le k\le n}$, we denote
$\epsilon_n$ the mesh size defined by $\epsilon_n=t_{k+1}-t_k = c
n^{-\mathfrak{a}}$ with a positive constant $c$ and $\mathfrak{a}\in
(0{,}1)$. We may also observe that $n=t_n^\gamma$ (up to a
multiplicative constant) with $\gamma=1/(1-\mathfrak{a})>1$. The
Nadaraya--Watson estimator of $b(x)$ with the bandwidth~$h$ is defined
at time $t_n$ by
\[
\hat b^{\mathrm{\sc{nw}}}_{t_n,h}(x) = \frac{\sum_{k=0}^{n-1} (t_n-t_k)^{1-2H} N  ( {(X_{t_k}-x)}/h  ) ( X_{t_{k+1}}-X_{t_k} )}%
{ \sum
_{k=0}^{n-1} (t_n-t_k)^{1-2H}
N ( {(X_{t_k} -x)}/h ) ( t_{k+1}-t_k )},
\]
where the kernel $N$ is a positive regular function with support in
$[-1{,}+1]$. It is obtained via a discretization of the continuous
version of the estimator (see Definition \ref{def1}). Some deviation
probability bounds are proved for the continuous and the discrete
version of the Nadaraya--Watson estimator in Theorem \ref{th-dev}.
If we assume that the drift has polynomial growth of order $\mm$ and
satisfies a one-sided dissipative Lipschitz, then we obtain ergodic
properties for the solution of \eqref{eds} (see Proposition \ref
{ergo}). This means that there exists a random variable $\bar X$ such
that the solution of equation \eqref{eds} converges for $t\to\infty$
to the stationary and ergodic process $(\bar X_t)_{t\ge0}=(\bar
X(\theta_t(\omega)))_{t\ge0}$ where $\theta_t$ is the appropriate
shift operator on the canonical probability space associated to the fBm.
Then we shall prove the consistency of the estimator:
\[
\hat b^{\mathrm{\sc{nw}}}_{t_n,h}(x) \mathop{{\hbox to 50pt{
\rightarrowfill}}}^{\mathrm{in\ probability}}_{n\to\infty,h\to 0} b(x)
\]
under the additional assumption that the number of approximation points
satisfies $n=t_n^\gamma$ with $\gamma>1+\mm H^2$ (see Remark \ref
{talk} for a discussion about the dependence between $n$, $\mm$ and
$H$) and another assumption on the non-degeneracy of the stationary
solution (see Hypothesis \ref{reg-inv}).

Similar results are obtained for the locally linear estimator that we
do not present in this \hyperref[sec1]{Introduction} because it would require further
and heavy notation. Nevertheless, the approach is identical. We first
construct a continuous time version of the estimator in Definition \ref
{def2} and some deviation probability bounds are obtained in Theorem
\ref{th-dev-ll}. Then a discrete version is proposed (see Definition
\ref{hatn}) and the consistency is obtained in
Theorem \ref{consist-ll} which is the other main result of this work.

The remainder of this paper is structured as follows. In Section \ref
{prelim}, we give some notation and we state our main assumptions.
Then we recall the link between $B^H$ and the fundamental martingale
for which classical stochastic calculus is available. This allows us to
introduce a new observable process having a semi-martingale
decomposition (see \eqref{YB}). Then we enunciate the ergodic
properties of the solution. The ergodic property under discrete
observations as it is stated in Proposition \ref{ergo} is new under
our assumptions on the discretization procedure. Its proof is postponed
in Appendix \ref{ergo-proof}. Section \ref{nada} is devoted to the
Nadaraya--Watson estimator of the drift whereas the study of the
locally-linear estimator is done in Section \ref{loclin}. In the two
sections cited above, we state deviation probability bounds and
consistency of the estimators under continuous and discrete
observations. Some proofs related to the locally-linear estimator are
gathered in Section \ref{proofs}. Finally, we shall make use of a
Fernique's type lemma that is stated and proved in Appendix \ref{app1}.
%
%s2 #&#
\section{Preliminaries}\label{prelim}
We consider a complete probability space $(\Omega, \mathcal F,
\mathbf P)$ on which a one dimensional fractional Brownian motion
$B^H$ is defined. We denote $\mathcal F_t
=\sigma(B^H_s,s\le t)$ the $\sigma$-field generated by $B^H$ completed
with respect to $\mathbf P$.

In a first subsection, we give some notation and we state our
assumptions. Then we indicate how to associate to the observed process
an auxiliary
semi-martingale which is appropriate for the statistical analysis. For
this, we will resume the notation of \cite{kl,nvv}.
Thereafter, the ergodic properties of the stochastic differential
equation \eqref{eds} will be discussed under ad-hoc
assumption on the drift.

%s2.1 #&#
\subsection{Notation and assumptions}
In all the sequel, we use the following notation. If $f$ and $g$ are
two functions form ${\mathbf R}$ to ${\mathbf R}$, we write
$f(t)\succeq g(t) $ when there exists a constant $K$ such that
$f(t)/g(t)\ge K$.
When the ratio of $f$ and $g$ is constant, we write $f(t)\asymp g(t)$.

The drift $b$ may satisfy one or several items of the following hypothesis.\
%
%hy1 #&#
\begin{hyp}\label{hypb} %
\begin{longlist}[(a)]
\item[(a)] (Local regularity). For any $x$, $b$ is
locally H\"{o}lder of order $\bb$ in the point $x$: there exists $L_x$
such that
\[
\bigl|b(y)-b\bigl(y'\bigr)\bigr|  \le L_x
\bigl|y-y'\bigr|
\]
for any $y,y'$ in a neighbourhood of $x$.
\item[(b)] (Global regularity). The drift term $b$
is continuously differentiable with a polyno\-mial growth condition on
its derivative and on $b$ itself:
there exists $c_b>0$ and $\mm\in\mathbf{N}$ such that
\[
\bigl|b(x)\bigr| +\bigl|b'(x)\bigr| \le c_b \bigl(1+|x|^\mm
\bigr) ,\qquad x\in\R .
\]
\item[(c)] (One-sided dissipative Lipschitz
condition). There exists a constant $L>0$ such that for
\[
\bigl( b(y)-b\bigl(y'\bigr) \bigr)\times\bigl(y-y'
\bigr)  \le-L \bigl|y-y'\bigr|^2 ,\qquad y,y'\in \R.
\]
\end{longlist}
\end{hyp}
We remark that the one-sided dissipative Lipschitz condition implies
that $b$ is Lipschitz with the same constant $L$.
It has been proved in \cite{nt} that there exists a unique solution to
equation \eqref{eds} under Hypotheses \ref{hypb}(b) and (c).

The kernel function that we shall need satisfies the following usual properties.
%
%hy2 #&#
\begin{hyp}\label{kernel}
The kernel function $N$ is continuously differentiable, non-negative
with support in $[-1{,}1]$.
Without loss of generality, we may assume that it is bounded by $1$.
\end{hyp}
Hypothesis \ref{kernel} is supposed to be fulfilled in all the rest of
this paper.

We also need the following notation concerning the discretization of
the time interval $[0,T]$:
%
%hy3 #&#
\begin{hyp}\label{tk}
For a given $n\in\N$, a time discretization $\{t_k\}_{0\le k\le n}$
is considered with equally spaced observation times $\epsilon
_{n}:=t_{k+1}-t_k\asymp n^{-\mathfrak{a}}$ with $\mathfrak{a}\in(0{,}1)$.

We observe that the number of approximation points $n$ is related to
the time horizon of the discrete observations $t_n$ by $n\asymp
t_n^\gamma$ with $\gamma= 1/(1-\mathfrak{a})>1$.
%For a given time horizon $T$ we set
%we shall take $n(T)\asymp T^\gamma$ with $\gamma>1$.
\end{hyp}
We remark that $\epsilon_{n}\asymp n^{-(\gamma-1)/\gamma}$. The
forthcoming discussions will be held by means of $\gamma$ instead of
$\mathfrak{a}$ because they lead to more readable expressions.
%
%s2.2 #&#
\subsection{Preliminaries on fractional Brownian motion}
It is difficult to work directly with the fBm $B^H$ because it is not a
semimartingale. Hence, we introduce some related
processes that will have nicer properties. For this purpose, let $w_H$
be the function defined by
%
%e4 #&#
\begin{equation}
\label{w} w_H(t,s)  = c_H s^{1/2-H}(t-s)^{1/2-H}
\1_{(0{,}t)}(s),
\end{equation}
where $c_H = (2H \Gamma(3/2-H)\Gamma(H+1/2))^{-1}$. Thanks to \cite
{kl,nvv}, the process $M^H=(M^H_t)_{t\ge0}$ defined by
%
%e5 #&#
\begin{equation}
\label{M} M^H_t = \int_0^t
w_H(t,s) \,\d B_s^H
\end{equation}
is a centered Gaussian process with independent increments. Its
variance function is given by
\[
\mathbf E \bigl(\bigl(M^H_t\bigr)^2 \bigr) =
\frac{\Gamma(3/2-H)}{2H\Gamma(3-2H)\Gamma(H+1/2)} t^{2-2H} := \lambda_H t^{2-2H} .
\]
Thus, $(M^H_t)_{t\ge0}$ is a martingale (called the fundamental
martingale in \cite{kl}).
The natural filtration of the martingale $M^H$ coincides with the
natural filtration of the fBm $B^H$. Finally, the process
$B=(B_t)_{t\ge0}$ defined by
\[
B_t  = \frac{1}{\sqrt{\lambda_H(2-2H)}} \int_0^t
s^{H-1/2} \,\d M^H_s
\]
is a standard Brownian motion that generates the same filtration as
$B^H$ and $M^H$.
The inverse relationship will also be helpful:
%
%e6 #&#
\begin{equation}
\label{B} M_t^H  =\bigl(\lambda_H(2-2H)
\bigr)^{1/2} \int_0^t s^{1/2-H}
\,\d B_s .
\end{equation}
We introduce the observable process $Y=(Y_t)_{t\ge0}$ defined by
%
%e7 #&#
\begin{eqnarray}
\label{Y} Y_t & =&x_0+ \int_0^t
w_H(t,s)\,\d X_s
\nonumber
\\[-8pt]
\\[-8pt]
\nonumber
& = &x_0 + \int_0^t
w_H(t,s)b(X_s) \,\d s + \int_0^tw_H(t,s)\,\d B_s^H.
\end{eqnarray}
By \eqref{M} and \eqref{B}, we have the following alternative expressions:
%
%e8 #&#
\begin{eqnarray}\label{YB}
Y_t& =& x_0 + \int_0^t
w_H(t,s)b(X_s) \,\d s + M_t^H
\nonumber
\\[-8pt]
\\[-8pt]
\nonumber
&= &x_0 + \int_0^t
w_H(t,s)b(X_s) \,\d s + \bigl(\lambda _H(2-2H)
\bigr)^{1/2}\int_0^t
s^{1/2-H}\,\d B_s .
\end{eqnarray}
In order to use the martingale $M^H$, we remark that
\[
w_H(t,s) \,\d s = \frac{c_H}{(2-2H) \lambda_H} (t-s)^{1/2-H}s^{H-1/2}
\,\d \bigl\langle M^H\bigr\rangle_s,
\]
thus if we let
\[
\tilde w(t,s)  = \frac{c_H}{(2-2H) \lambda_H} (t-s)^{1/2-H}s^{H-1/2}
\1_{(0{,}t)}(s)
\]
we may write
%
%e9 #&#
\begin{equation}
\label{YM2} Y_t = x_0 + \int_0^t
\tilde w_H(t,s)b(X_s) \,\d \bigl\langle M^H\bigr
\rangle_s + M_t^H .
\end{equation}
The above representations will be the starting point of the
construction of our estimators.
%%---------------------------------------------------------------------------------------------
%%---------------------------------------------------------------------------------------------
%s2.3 #&#
\subsection{Ergodic properties of the stochastic differential equation}

In this subsection, we give details on the ergodic properties
of the fractional SDE \eqref{eds}. We use the results of Section 4 in
\cite{kn}, and we borrow the presentation of \cite{nt}.
However, we repeat it for conciseness and we give some precisions.

Without loss of generality, we work on the canonical probability space
$(\Omega,\mathcal F,\P)$
associated to a fBm $B^H:=(B_t^H)_{t\in\R}$ defined on $\R$
entirely. This means that $B^H$ is a zero mean
Gaussian process having the variance function equals to $\E
(|B_t^H-B_s^H|^2) = |t-s|^{2H}$ for any $s,t\in\R$.
The Wiener space $\Omega$ is the topological space $ C_0(\R;\R)$
equipped with the compact open topology and $\mathcal F$ is the
associated Borel $\sigma$-algebra.
The measure $\P$ is the distribution of the fBm $B^H$ which now
corresponds to the evaluation process $B_t^H(\omega) = \omega(t)$ for
$t\in\R$.
The law of the two-sided fBm is invariant to the shift operators with
increment $t\in\R$. In other word, the operator
$\theta_t$ defined from $\Omega$ to $\Omega$ by $\theta_t \omega
(\cdot) = \omega(\cdot+t)-\omega(\cdot)$ is such that
the shifted process $(B_s(\theta_t\cdot))_{s\in\R}$ is again a fBm.

Moreover for all integrable real valued random variable $F$ it holds
%
%e10 #&#
\begin{equation}
\label{err}  \lim_{t\to\infty} \frac{1}T \int
_0^T F \bigl( \theta_t(\omega )
\bigr) \,\d t = \E(F)\qquad \mbox{$\P$-almost surely.}
\end{equation}

The ergodic properties of \eqref{eds} will hold under the assumption
that the drift $b$
satisfies the polynomial growth condition \ref{hypb}(b) and the
one-sided dissipative Lipschitz condition \ref{hypb}(c).
Under these hypotheses, there exists a random variable $\bar X$ with
finite moments of any order, and such that
%
%e11 #&#
\begin{equation}
\label{err2} \lim_{t\to\infty} \bigl| X_t(\omega) - \bar X
\bigl(\theta_t(\omega) \bigr) \bigr| = 0
\end{equation}
for $\P$-almost-all $\omega\in\Omega$. Thus the solution of
equation \eqref{eds} converges when $t$ goes to infinity to a
stationary and ergodic process $(\bar X_t)_{t\ge0}$ defined by $\bar
X_t(\omega)=\bar X(\theta_t(\omega))$.
By \cite{h,h-pillai} the law of $(\bar X_t)_{t\ge0}$ coincides with
the attracting invariant measure for the solution of \eqref{eds}.
The next proposition will be crucial when we will study consistency of
our estimators.
%
%pr1 #&#
\begin{prop}\label{ergo}
Assume that Hypotheses \textup{\ref{hypb}(b)} and \textup{(c)} are true.
Consider a continuously differentiable function
$\ffi$ such that
%
%e12 #&#
\begin{equation}
\label{hyp-ffi}  \bigl|\ffi(y)\bigr|+\bigl|\ffi'(y)\bigr| \le c_\ffi
\bigl(1+|y|^\ppp\bigr) ,\qquad  y\in{\mathbf R}
\end{equation}
for some $c_\ffi>0$ and $\ppp\in\mathbf{N}$.
\begin{longlist}[(ii)]
\item[(i)]We have
%
%e13 #&#
\begin{equation}
\label{ergo-lim1} \lim_{T\to\infty} \frac{1}T \int
_0^T \ffi(X_s)\,\d s =\E \bigl( \ffi (
\bar X) \bigr) \qquad\mbox{$\P$-a.s.}
\end{equation}
\item[(ii)]If $\gamma> 1+(\mm^2+\ppp)  H$ and
$\gamma> \ppp+1$ then
%
%e14 #&#
\begin{equation}
\label{ergo-lim2} \lim_{n\to\infty} \frac{1}{t_n} \int
_0^{t_n} \Biggl\{ \sum
_{k=0}^{n-1} \ffi(X_{t_k})
\1_{[t_k{,}t_{k+1})}(s) \Biggr\} \,\d s  = \E \bigl( \ffi(\bar X) \bigr)\qquad \mbox{$\P$-a.s.},
\end{equation}
where the observation times are defined in Hypothesis \ref{tk}.
\end{longlist}
%
% \lim_{T\to\infty} \frac{1}T \int_0^T \ffi(X_s)\,\d s & = \E\big( \ffi(
% \label{ergo-lim2}
% \lim_{n\to\infty} \frac{1}{t_n} \int_0^{t_n} \mbox{$\left\{
%&= \E\big( \ffi(\bar X)\big)  \P-\mbox{a.s.}
\end{prop}
Let us make the following remarks and comments about the above proposition.

%re1 #&#
\begin{remi}
The proof of this result is partially contained in Proposition 2.3 and
Lemma 3.1 in~\cite{nt}.
But in our result, we have a condition on the number of approximation
points that depends on $H$ and on the
degrees of polynomial growth $\mm$ and $\ppp$. We think that it is
not possible to get rid of the fact that
$n\asymp t_n^\gamma$ with $\gamma>1+\max ( (\mm^2 +\ppp)H ,  \ppp )$.
\end{remi}
A proof of this result is proposed in Appendix \ref{ergo-proof}.
%
%re2 #&#
\begin{remi}\label{talk}
The above result is valid for a wide class of drift function since $\mm
\in\{0,1,2,\ldots\}$.
We cover the case of bounded function as well than the case of linear
and polynomial growing functions. Such a
remark is valid for the function $\ffi$.

Assume that $\ffi$ is bounded together wit hits derivative ($\ppp=0$).
The condition on $\gamma$ becomes $\gamma> 1+\mm^2 H$ and thus the
number of approximation points is related to the polynomial growth
order $\mm$
and to the Hurst parameter $H$. When $\mm$ is fixed, we need more
points in the time discretization when $H$ grows.
This is intuitively correct since the trajectories becomes more regular
when $H$ increases. Thus, the process is less oscillating and it is
necessary to observe more often the diffusion in order to insure that
$X$ visits very often any
neighbourhood of any fixed point.

When $H$ is fixed, the number of approximation points is growing with
the polynomial growth coefficients $\mm$ and $\ppp$.
It is surely related to the speed of convergence of the process $X$ to
the stationary ergodic process $\bar X$.
It is intuitive to think that when we let the drift coefficient behaves
like of polynomial function,
the convergence must be slower when the degree is big. To our
knowledge, such investigation has not yet been carried out.
\end{remi}
%
%%---------------------------------------------------------------------------------------------
%%---------------------------------------------------------------------------------------------
%%---------------------------------------------------------------------------------------------
%%---------------------------------------------------------------------------------------------
%%---------------------------------------------------------------------------------------------
%%---------------------------------------------------------------------------------------------
%%---------------------------------------------------------------------------------------------
%%---------------------------------------------------------------------------------------------
%%---------------------------------------------------------------------------------------------
%%---------------------------------------------------------------------------------------------
%%---------------------------------------------------------------------------------------------
%%---------------------------------------------------------------------------------------------
%%---------------------------------------------------------------------------------------------
%s3 #&#
\section{The Nadaraya--Watson type estimator}
\label{nada}
%%---------------------------------------------------------------------------------------------
%%---------------------------------------------------------------------------------------------
%%---------------------------------------------------------------------------------------------
%%---------------------------------------------------------------------------------------------
%%---------------------------------------------------------------------------------------------
%%---------------------------------------------------------------------------------------------
Our first method for estimating the value of the drift $b$ in a fixed
point $x\in\R$ is inspired of the Nadaraya--Watson
kernel regression. We construct this estimator in the following
subsection. Thereafter, some deviation probability bounds are given and
finally, the consistency will be stated under the ergodicity assumption.
%
%s3.1 #&#
\subsection{Construction and decomposition of the Nadaraya--Watson estimator}
\label{nada11}
First of all, we assume that the whole trajectory $(X_t)_{0\le t\le T}
$ is observed between the times $0$ and~$T$.
We will discuss a discretized version of our estimator in a moment.

The construction of a Nadaraya--Watson estimator is based on a simple
idea. First, we think that the drift $b$ is a constant function, that
is $b(x)=\theta$ for any $x$.
Hence an estimator of $\theta$ is an estimator of $b(x)$. We denote
\[
X_t^\theta= x_0 +\int_0^t
\theta \,\d s + B_t^H .
\]
Similarly to \eqref{YM2}, we introduce the observable process
$Y^\theta=(Y^\theta_t)_{t\ge0}$
\[
Y^\theta_t  = x_0 + \int_0^t
\tilde w_H(t,s) \theta \,\d \bigl\langle M^H\bigr
\rangle_s + M^H_t .
\]
The unknown parameter $\theta$ can be estimated by the least squares
method (see, e.g., \cite{lbm}).
The least squares estimator of $\t$ obtained at time $t$ is given by
\[
\hat\theta(t) = \frac{\int_0^t \tilde w_H(t,s)\,\d Y_s^\t}{\int_0^t
(\tilde w_H(t,s) )^2 \,\d \langle M^H\rangle_s} .
\]
We denote
%
%e15 #&#
\begin{equation}
\label{alfa}  \alpha_H=\frac{c_H}{\sqrt{\lambda_H (2-2H)}} ,
\end{equation}
and we remark that
\[
\int_0^t \bigl(\tilde w_H(t,s)
\bigr)^2 \,\d \bigl\langle M^H\bigr\rangle_s =
\int_0^t {\alpha}_H^2
(t-s)^{1-2H} \,\d s .
\]
Thus, we obtain the alternative representation of $\hat\theta(t)$:
\[
\hat\theta(t) = \frac{\int_0^t \tilde w_H(t,s)\,\d Y_s^\t}{\int_0^t
{\alpha}_H^2 (t-s)^{1-2H} \,\d s} .
\]
In the context of our fractional diffusion \eqref{eds}, the drift
$b$ is not constant. Hence, we approximate it by a constant function
$\theta$ in a neighbourhood $[x-h,x+h]$ of the
point $x$. For this purpose, we consider a kernel function $N$
satisfying Hypothesis \ref{kernel}.
The above discussion leads to the following definition of an estimator
of $b(x)$
by means of the observable process $Y$.
%
%de1 #&#
\begin{defi}\label{def1}
The Nadaraya--Watson estimator of the drift $b$ in a point $x$ with the
bandwidth $h$ is defined at time $t$ by
%
%e16 #&#
\begin{eqnarray}
\label{defnw} \hat b^{\mathrm{\sc{nw}}}_{{t,h}}(x)  = \frac{\int_0^t ({\alpha_H^2}/{c_H}) (t-s)^{1/2-H} s^{H-1/2} N
 ({(X_s -x)}/h  ) \,\d Y_s}%
{\int_0^t \alpha_H^2
(t-s)^{1-2H} N ( {(X_s -x)}/h ) \,\d s}
\end{eqnarray}
with the convention that $a/0:=0$. Equivalently, the more classical
expression holds
%
%e17 #&#
\begin{eqnarray}
\label{defnw2} \hat b^{\mathrm{\sc{nw}}}_{{t,h}}(x)  = \frac{\int_0^t \alpha_H^2
(t-s)^{1-2H} N  ( {(X_s -x)}/h  ) \,\d X_s}%
{\int_0^t \alpha_H^2
(t-s)^{1-2H} N ( {(X_s -x)}/h ) \,\d s} .
\end{eqnarray}
\end{defi}
Using the representation \eqref{YB} and the fact that $N\le1$, we
notice that the stochastic integral in~\eqref{defnw} is well defined.
The integral with respect to the process $X$ in \eqref{defnw2} is just
an alternative writing of the one with respect to $Y$ in \eqref
{defnw}. Moreover, starting from \eqref{defnw} and using \eqref{YB},
\eqref{w} and~\eqref{alfa} we may express our estimator as
\begin{eqnarray*}
\hat b^{\mathrm{\sc{nw}}}_{{t,h}}(x) &=&  \frac{\int_0^t \alpha_H^2 (t-s)^{1-2H} N  ( {(X_s -x)}/h
) b(X_s) \,\d s}%
{\int
_0^t \alpha_H^2
(t-s)^{1-2H} N ( {(X_s -x)}/h ) \,\d s} \\
&&{}+
\frac{\int_0^t \alpha_H (t-s)^{1/2-H} N  (
{(X_s -x)}/h  ) \,\d B_s}%
{\int_0^t
\alpha_H^2 (t-s)^{1-2H} N (
{(X_s -x)}/h ) \,\d s}.
\end{eqnarray*}
Then we obtain the following decomposition of the error:
%
%e18 #&#
\begin{equation}
\label{decomp-nw} \hat b^{\mathrm{\sc{nw}}}_{{t,h}}(x) = b(x) +
\xi_{x,h}(X_t) + r_{x,h}^{\mathrm{loc}}(X_t),
\end{equation}
where
\begin{eqnarray*}
\xi_{x,h}(X_t) & = &\frac{\int_0^t \alpha_H (t-s)^{1/2-H}
N  ( {(X_s -x)}/h  ) \,\d B_s}%
{\int
_0^t \alpha_H^2
(t-s)^{1-2H} N ( {(X_s -x)}/h ) \,\d s} ,
\\
r_{x,h}^{\mathrm{loc}}(X_t)& =& \frac{\int_0^t (t-s)^{1-2H} N  ( {(X_s -x)}/h  )  [ b(X_s) -b(x)  ] \,\d s}%
{\int_0^t (t-s)^{1-2H} N (
{(X_s -x)}/h ) \,\d s}.
\end{eqnarray*}
There are two kinds of errors in \eqref{decomp-nw}. The first one is a
stochastic one (the term $\xi_{x,h}(X_t)$).
The second one ($r_{x,h}^{\mathrm{loc}}(X_t)$) represents the accuracy
of the
local approximation of $b$ by a constant function in a neighbourhood of
the point $x$.

From a practical point of view, the real interest is the case when the
observed data are discrete.
So we provide now an effective estimation procedure. We assume that the
process $(X_t)_{0\le t\le T}$ is observed at
times $(t_k)_{0\le k\le n}$ (see Hypothesis \ref{tk}).
We discretize the expression of $\hat b^{\mathrm{\sc
{nw}}}_{{t,h}}(x)$ given in~\eqref{defnw2} by Riemann sums as
\begin{eqnarray*}
\hat b^{\mathrm{\sc{nw}}}_{t_n,h}(x)  = \frac{
\sum_{k=0}^{n-1} (t_n-t_k)^{1-2H} N  ( {(X_{t_k}
-x)}/h  ) ( X_{t_{k+1}}-X_{t_k} )}%
{ \sum
_{k=0}^{n-1} (t_n-t_k)^{1-2H}
N ( {(X_{t_k} -x)}/h ) ( t_{k+1}-t_k )}.
\end{eqnarray*}
In order to have a decomposition of the error, we consider the simple
process $(Q^n_s)_{s\ge0}$
defined by
\[
Q^n_s = \sum_{k=0}^{n-1}
(t_n-t_k)^{1-2H} N \biggl( \frac
{X_{t_k} -x}h
\biggr) \1_{[t_k{,}t_{k+1})}(s) .\
\]
With the help of equations \eqref{Y} and \eqref{YB}, we may rewrite
$\hat b^{\mathrm{\sc{nw}}}_{t_n,h}(x)$ as
\begin{eqnarray*}
\hat b^{\mathrm{\sc{nw}}}_{t_n,h}(x) & =& \frac{1}{\int_0^{t_n}
Q^n_s \,\d s} \int
_0^{t_n} Q^n_s
\,\d X_s
\\
&= &\frac{1}{\int_0^{t_n} Q^n_s \,\d s} \int_0^{t_n}
\frac{Q^n_s}{w_H(t,s)} \,\d Y_s
\\
& =& \frac{1}{\int_0^{t_n} Q^n_s \,\d s} \biggl( \int_0^{t_n}
Q^n_s b(X_s) \,\d s+\bigl(\lambda_H(2-2H)
\bigr)^{1/2}\int_0^{t_n}
\frac
{Q^n_s}{w_H(t,s)} s^{1/2-H} \,\d B_s \biggr).
\end{eqnarray*}
Thus, we obtain a similar decomposition than \eqref{decomp-nw}
%
%e19 #&#
\begin{equation}
\label{decomp-n} \hat b^{\mathrm{\sc{nw}}}_{t_n,h}(x) - b(x) =
\xi_{x,h}\bigl(X^n_{t_n}\bigr) +
r_{x,h}^{\mathrm{loc}}(X_{t_n}) +r_{x,h}^{\mathrm{traj}}(X_{t_n})
\end{equation}
with
\begin{eqnarray*}
\xi_{x,h}(X_{t_n}) & =& \frac{\int_0^{t_n} \alpha_H
(t-s)^{H-1/2}  Q^n_s \,\d B_s}{\int_0^{t_n} \alpha_H^2  Q^n_s \,\d s} ,
\\
r_{x,h}^{\mathrm{loc}}(X_{t_n}) & =&
\frac{\int_0^{t_n} \sum_{k=0}^{n-1} (t_n-t_k)^{1-2H} N  ( {(X_{t_k} -x)}/h
)   ( b(X_{t_k})-b(x) )\1_{[t_k{,}t_{k+1})}(s) \,\d s}{\int_0^{t_n} Q^n_s \,\d s} ,
\\
r_{x,h}^{\mathrm{traj}}(X_{t_n}) & =&
\frac{\int_0^{t_n} \sum_{k=0}^{n-1} (t_n-t_k)^{1-2H} N  ( {(X_{t_k} -x)}/h
)   ( b(X_{s})-b(X_{t_k}) )\1_{[t_k{,}t_{k+1})}(s) \,\d s}{\int_0^{t_n} Q^n_s \,\d s} .
\end{eqnarray*}
We remark that $r_{x,h}^{\mathrm{loc}}(X_{t_k})$ represents again the
accuracy of the
local approximation of $b$ by a constant function in a neighbourhood of
the point $x$, but only in the discrete times
$(t_k)_{0\le k\le n}$. It is worth to notice that a new term is
involved: $r_{x,h}^{\mathrm{traj}}(X_{t_n})$. It represents the error made
when one proceed to the discretization of the continuous process
$(X_s)_{s\ge0}$.

In the next subsection, we study deviation probability bounds for $\hat
b^{\mathrm{\sc{nw}}}_{t}(x)$ and
$\hat b^{\mathrm{\sc{nw}}}_{t,n}(x)$.

%s3.2 #&#
\subsection{Deviation probability}\label{ji}
In order to study the error from a probabilistic point of view, we need
to introduce
for some $\rho>0$ and $\beta>0$ the random sets
\begin{eqnarray*}
\A^{\mathrm{\sc{nw}}}_{t,h} & =& \biggl\{ \int_0^t
\alpha _H (t-s)^{1-2H} N \biggl( \frac{X_s -x}h \biggr)
\,\d s\ge\rho t^{1-H+\beta} \biggr\} \quad\mbox{and}
\\
\A^{\mathrm{\sc{nw}}}_{t_n,h} & =& \Biggl\{ \int_0^{t_n}
\alpha_H\sum_{k=0}^{n-1}
(t_n-t_k)^{1-2H} N \biggl( \frac{X_{t_k} -x}h
\biggr) \1_{[t_k{,}t_{k+1})}(s)\,\d s \ge\rho t_n^{1-H+\beta} \Biggr\} .
\end{eqnarray*}
Some properties of the Nadaraya--Watson estimator are stated in the
following theorem conditionally on the above events.
%
%th2 #&#
\begin{theo}\label{th-dev}
Under Hypothesis \textup{\ref{hypb}(a)}, when the trajectory is continuously
observed, we have for any $\zeta>0$:
%
%e20 #&#
\begin{eqnarray}
\label{poo2}  \mathbf P \bigl( \bigl| \hat b^{\mathrm{\sc{nw}}}_{t,h}(x)- b(x)
\bigr|\ge L_x h^\bb+ \zeta , \A^{\mathrm{\sc
{nw}}}_{t,h}
\bigr) \le2\exp \bigl( -\rho^2 (1-H) \zeta^2
t^{2\beta} \bigr) .
\end{eqnarray}
We assume that $b$ satisfies Hypotheses \textup{\ref{hypb}(b)} and
\textup{(c)}.
There exists $\tau_0\ge1$, $c_{x,h,L}>0$\footnote{See \eqref
{tau0}, \eqref{cxhl} in the proof for an explicit expression of $\tau
_0$ and $c_{x,h,L}$.}
and a constant $c_{H,\gamma}$ such that for any $t_n\ge\tau_0$, the
following conditional deviation probability bound holds:
%
%e21 #&#
\begin{eqnarray}
\label{eqdidi} &&\mathbf P \bigl( \bigl| \hat b^{\mathrm{\sc{nw}}}_{t_n,h}(x)-
b(x)\bigr|\ge L h+c_{x,h,L} \epsilon_n + \zeta ,
\A^{\mathrm{\sc
{nw}}}_{t_n,h} \bigr)
\nonumber
\\[-8pt]
\\[-8pt]
\nonumber
&&\quad \le2\exp \biggl( -\frac{\rho^2  (1-H)  \zeta^2}{16
\alpha_H^2} t_n^{2\beta} \biggr)+
c_{H,\gamma} t_n^{{H(\gamma-1)}/{(\gamma+1)}} \exp \biggl( -\frac{\zeta}{4L}
t_n^{{(2H(\gamma-1))}/{(\gamma+1)}} \biggr) .
\end{eqnarray}
\end{theo}
It is interesting that the above results about the quality of our
estimation are non-asymptotic and do not require any ergodic or mixing
properties of the observed process. Clearly the event $\A^{\mathrm
{\sc{nw}}}_{t_n,h}$ is completely
determined by the observed values of the trajectory of $X$. It is
therefore always possible to check whether the path belongs or not to
this set. If it is not the case, we are not able to guarantee a
reasonable quality for the estimation of $b(x)$.

In the following remark, we discuss the rate of our approximation.
%
%re3 #&#
\begin{remi}\label{rrr} %
\begin{enumerate}
\item If we choose a time dependent bandwidth $h_t$ such that $h_t^\bb
\asymp L_x^{-1} t^{-\beta/2}$,
then the rate of estimation is of order $t^{-\beta/2}$:
\[
 \mathbf P \bigl( \bigl| \hat b^{\mathrm{\sc{nw}}}_{t,h_t}(x)- b(x) \bigr| \succeq
t^{-\beta/2} , \A^{\mathrm{\sc
{nw}}}_{t,h_t} \bigr) \le2\exp \bigl( -
\rho^2 (1-H) t^{\beta} \bigr) .
\]
\item In the discrete case, for a fixed $\beta>0$, we consider:
\begin{itemize}
\item$h_{n}$ a time dependant bandwidth with $h_{n} \asymp L^{-1}
t_n^{-\beta/2}$;\vadjust{\goodbreak}
\item$\gamma=(4H+\beta)/(4H-\beta)$;
\item$\epsilon_n=t_n/n\asymp t_n^{-(\gamma-1)}$ with $\gamma-1 =
2\beta/(4H-\beta)>\beta/2$.
\end{itemize}
Then the approximation rate is again of order
$t_n^{-\beta/2}$ since \eqref{eqdidi} implies
\begin{eqnarray*}
\mathbf P \bigl( \bigl| \hat b^{\mathrm{\sc{nw}}}_{t_n,h}(x)- b(x)\bigr|\succeq
t^{-\beta/2} , \A^{\mathrm{\sc{nw}}}_{t_n,h_n} \bigr)  \preceq\exp \bigl(
-C_{\rho,L,H} t_n^{\beta} \bigr)+ t_n^{\mu_1}
\exp \biggl( -\frac{t_n^{\mu_2}}{4L} \biggr) ,
\end{eqnarray*}
with $\mu_1,\mu_2>0$.
\end{enumerate}
\end{remi}
The stochastic integral that appears in the expression of $\xi
_{x,h}(X_t)$ is a
fractional martingale (so called in \cite{hns}). In order to
study the asymptotic behaviour of the Nadaraya--Watson estimators, we
need asymptotic properties of this
fractional martingale. This will be done thanks to a straightforward
exponential inequality for
this kind of stochastic integral. One refers to \cite{s-expo}
for related results on exponential inequalities for fractional martingales.
%
%le3 #&#
\begin{lemme}
\label{exp}
We consider $K = (K_s)_{s\ge0}$, an adapted process
such that for a positive function $v$
\[
\sup_{0\le u\le t}\int_0^u
(t-s)^{1-2H} |K_s|^2\,\d s \le v(t) .
\]
Then for any $\zeta\ge0$ it holds that
%
%e22 #&#
\begin{equation}
\label{inegexp} \mathbf P \biggl( \biggl\llvert \int_0^t
(t-s)^{1/2-H} K_s \,\d B_s \biggr\rrvert \ge \zeta
\biggr)  \le2\exp \biggl( - \frac{\zeta^2}{2  v(t)} \biggr).
\end{equation}
\end{lemme}
\begin{pf}
For a fixed time $t$, we consider the true martingale $(Z^t_u)_{0\le
u\le t}$ defined by
\[
Z^t_u = \int_{0}^u
(t-s)^{1/2-H} K_s \,\d B_s .
\]
Here $t$ is consider as a fixed parameter for the martingale $Z^t$.
It holds that
\[
\bigl\langle Z^t\bigr\rangle_u  = \int
_0^u (t-s)^{1-2H} |K_s|^2\,\d s
\le v(t) .
\]
The classical exponential inequality (see \cite{ry}, Exercice 3.16,
Chapter 4) implies the result.
\end{pf}
Now we can prove Theorem \ref{th-dev}.
\begin{pf*}{Proof of Theorem \ref{th-dev}}
The proof is divided in several steps.
%pa3.2.subsubsection.1 #&#

\textit{Step \textup{1:} Proof of \textup{\protect\eqref{poo2}}}.
We use the decomposition \eqref{decomp-nw}. Obviously we have the
following estimation
%
%e23 #&#
\begin{equation}
\label{rh} \bigl|r_{x,h}^{\mathrm{loc}}\bigr|  \le L_x
h^\bb .
\end{equation}
Thanks to the exponential inequality \eqref{inegexp} of Lemma \ref
{exp} we have for any $\zeta>0$
%
%e24 #&#
\begin{eqnarray}
\label{xi-nw}
&&\mathbf P \bigl( \bigl|\xi_{x,h}(X_t)\bigr|\ge\zeta ,
\A^{\mathrm{\sc
{nw}}}_{t,h} \bigr) \nonumber\\
&&\quad \le\mathbf P \biggl( \int
_0^t (t-s)^{1/2-H} N \biggl(
\frac{X_s -x}h \biggr) \,\d B_s \ge\rho \zeta t^{1-H+\beta} \biggr)
\\
&&\quad \le2\exp \bigl( -\rho^2 (1-H) \zeta^2 t^{2\beta}
\bigr).\nonumber
\end{eqnarray}
By \eqref{decomp-nw}, \eqref{rh} and \eqref{xi-nw} we obtain
\begin{eqnarray*}
\mathbf P \bigl( \bigl| \hat b^{\mathrm{\sc{nw}}}_{t,h}(x)- b(x)\bigr|\ge
L_xh^\bb+ \zeta , \A^{\mathrm{\sc{nw}}}_{t,h}
\bigr) & \le&\mathbf P \bigl( \bigl|\xi_{x,h}(X_t)\bigr|+\bigl|r_{x,h}^{\mathrm{loc}}(X_t)\bigr|
\ge L_x h^\bb+ \zeta , \A^{\mathrm{\sc{nw}}}_{t,h}
\bigr)
\\
&\le&\mathbf P \bigl( \bigl|\xi_{x,h}(X_t)\bigr|\ge\zeta
, \A^{\mathrm{\sc
{nw}}}_{t,h} \bigr)
\\
& \le&2\exp \bigl( -\rho^2 (1-H) \zeta^2
t^{2\beta} \bigr) ,
\end{eqnarray*}
and \eqref{poo2} is proved.

%pa3.2.subsubsection.2 #&#
\textit{Step \textup{2:} Proof of \textup{\protect\eqref{eqdidi}}}.
We analyse separately the three terms in the decomposition \eqref
{decomp-n}. We begin with $\xi_{x,h}(X_{t_n})$ and we
write
\begin{eqnarray*}
&& \P \bigl( \bigl|\xi_{x,h}(X_{t_n})\bigr|\ge\zeta , \mathcal
A^{\mathrm
{\sc{nw}}}_{t_n,h} \bigr)
\\
& &\quad= \P \biggl( \biggl\llvert \int_0^{t_n}
\alpha_H (t-s)^{H-1/2} Q^n_s
\,\d B_s \biggr\rrvert \ge\zeta\int_0^{t_n}
\alpha_H^2 Q^n_s \,\d s , \mathcal
A^{\mathrm{\sc{nw}}}_{t_n,h} \biggr)
\\
&&\quad \le\P \biggl( \biggl\llvert \int_0^{t_n}
\alpha_H (t-s)^{H-1/2} Q^n_s
\,\d B_s \biggr\rrvert \ge\zeta \rho t_n^{1-H+\beta}
\biggr).
\end{eqnarray*}
We fix $t_n$ and we consider the martingale $Z^n:=(Z^n_r)_{0\le r\le
t_n}$ defined by
\[
Z^n_r = \int_0^r
\alpha_H (t_n-s)^{H-1/2} Q^n_s
\,\d B_s .
\]
Since $0\le N\le1$, the quadratic variation of the martingale $Z^n$ satisfies:
\begin{eqnarray*}
\bigl\langle Z^n \bigr\rangle_r & =&
\alpha_H^2 \int_0^r
(t_n-s)^{2H-1} \Biggl\{ \sum_{k=0}^{n-1}
(t_n-t_k)^{2-4H} N^2 \biggl(
\frac{X_{t_k} -x}h \biggr) \1_{[t_k{,}t_{k+1})}(s) \Biggr\} \,\d s
\\
& \le& \alpha_H^2 \sum
_{k=0}^{n-1} \int_0^{t_n}
(t_n-s)^{2H-1} (t_n-t_k)^{2-4H}
\1_{[t_k{,}t_{k+1})}(s) \,\d s .
\end{eqnarray*}
When $H>1/2$, it holds $(t_n-s)^{2H-1}\le(t_n-t_k)^{2H-1}$ for $t_k\le
s< t_{k+1}$. Hence
\begin{eqnarray*}
\bigl\langle Z^n \bigr\rangle_{t_n} & \le&
\alpha_H^2 \sum_{k=0}^{n-1}
\int_0^{t_n} (t_n-t_k)^{1-2H}
\1_{[t_k{,}t_{k+1})}(s) \,\d s
\\
& \le&\alpha_H^2 \int_0^{t_n}
(t_n-s)^{1-2H} \,\d s
\\
&\le&\frac{\alpha_H^2}{2-2H} {t_n}^{2-2H}.
\end{eqnarray*}
For the second case when $H<1/2$, the inequality
$t_n-s>t_n-t_{k+1}=t_n-t_k-\Delta$ (valid for $t_k\le s< t_{k+1}$)
implies that
\[
(t_n-t_k)^{2-4H} \le(t_n-s)^{2-4H}
\biggl( 1 + \frac{\Delta
}{t_n-s} \biggr)^{2+4H} \le4 (t_n-s)^{2-4H}
.
\]
Therefore, we obtain
\[
\bigl\langle Z^n \bigr\rangle_{t_n} \le4
\alpha_H^2 \int_0^{t_n}
(t_n-s)^{1-2H} \,\d s = \frac{2  \alpha_H^2}{1-H} t_n^{2-2H}.
\]
By Lemma \ref{exp}, we conclude that
%
%e25 #&#
\begin{equation}
\label{arghh} \P \bigl( \bigl|\xi_{x,h}(X_{t_n})\bigr|\ge\zeta ,
\mathcal A^n_{t_n,h} \bigr) \le2\exp \biggl( -
\frac{\rho^2  (1-H)  \zeta
^2}{4  \alpha_H^2} t_n^{2\beta} \biggr) .
\end{equation}

Now we study the error term $r_{x,h}^{\mathrm{loc}}(X_{t_n})$ from
\eqref{decomp-n}.
The drift $b$ is Lipschitz by Hypothesis \ref{hypb}(c). So
we have
%
%e26 #&#
\begin{equation}
\label{r1} \bigl|r_{x,h}^{\mathrm{loc}}(X_{t_n})\bigr|  \le L h .
\end{equation}
The last term $r_{x,h}^{\mathrm{traj}}(X_{t_n})$ is more difficult to
handle. At first, we write
\begin{eqnarray*}
&&\bigl|r_{x,h}^{\mathrm{traj}}(X_{t_n})\bigr|
\\
&&\quad\le \frac{\int_0^{t_n} \sum_{k=0}^{n-1} (t_n-t_k)^{1-2H} N  ( {(X_{t_k} -x)}/h
)   L |X_{s}-X_{t_k}|  \1_{[t_k{,}t_{k+1})}(s) \,\d s}{\int_0^{t_n} Q^n_s \,\d s} .
\end{eqnarray*}
By equation \eqref{eds},
\[
X_s-X_{t_k} = \int_{t_k}^s
b(X_r)\,\d r + B_s^H-B_{t_k}^H
,
\]
and when $|X_{t_k}-x|\le h$ we may write for $t_k\le s< t_{k+1}$:
\begin{eqnarray*}
|X_s-X_{t_k}| & \le &\int_{t_k}^s
\bigl\{ \bigl|b( X_r)-b(X_{t_k})\bigr| + \bigl|b(X_{t_k})-b(x)\bigr|+\bigl|b(x)\bigr|
\bigr\} + \bigl|B_s^H-B_{t_k}^H \bigr|
\\
& \le&\epsilon_n \bigl\{ L h+c_b\bigl(1+|x|^\mm
\bigr) \bigr\} + L \int_{t_k}^s |
X_r-X_{t_k}| \,\d r + \epsilon_n^\h \bigl\|
B^H \bigr\| _{0,t_n,\mathfrak{h}} ,
\end{eqnarray*}
where we have denoted for $0<\h<H$:
\[
\bigl\| B^H \bigr\|_{0,t_n,\mathfrak{h}}=\sup_{0\le r,s\le t_n}
\frac
{|B_s^H-B_r^H|}{|s-r|^\mathfrak{h}} .
\]
When $\epsilon_n$ is small, the Gronwall inequality implies that for
any $t_k\le s< t_{k+1}$:
%
%e27 #&#
\begin{equation}
\label{nx} |X_s-X_{t_k}|  \le\epsilon_n \bigl
\{ L h+c_b\bigl(1+|x|^\mm\bigr) \bigr\}+
\epsilon_n^\h \bigl\| B^H \bigr\|_{0,t_n,\mathfrak{h}} .
\end{equation}
Therefore,
%
%e28 #&#
\begin{equation}
\label{rdeux} \bigl|r_{x,h}^{\mathrm{traj}}(X_{t_n}) \bigr|\le
c_{x,h,L} \epsilon_n + L \epsilon_n^\h
\bigl\| B^H \bigr\|_{0,t_n,\mathfrak{h}} ,
\end{equation}
with
%
%e29 #&#
\begin{equation}
\label{cxhl} c_{x,h,L}=L \bigl\{ L h+c_b
\bigl(1+|x|^\mm\bigr) \bigr\} .
\end{equation}

Now we are able to end the proof of \eqref{eqdidi}.
Starting from the decomposition \eqref{decomp-n}, using the
estimations \eqref{arghh}, \eqref{r1} and \eqref{rdeux}, we obtain
%
%e30 #&#
\begin{eqnarray}
\label{yu}
&&\mathbf P \bigl( \bigl| \hat b^{\mathrm{\sc{nw}}}_{t_n,h}(x)- b(x)\bigr|
\ge Lh+C_{x,h,L} \epsilon_n + \zeta , \A^{\mathrm{\sc{nw}}}_{t_n,h}
\bigr)
\nonumber
\\
&&\quad \le\mathbf P \bigl( \bigl|\xi_{x,h}(X_{t_n})\bigr|+\bigl|r_{x,h}^{\mathrm
{loc}}(X_{t_n})\bigr|+\bigl|r_{x,h}^{\mathrm{traj}}(X_{t_n})\bigr|
\ge Lh + C_{x,h,L} \epsilon_n+ \zeta , \A^{\mathrm{\sc{nw}}}_{t_n,h}
\bigr)
\nonumber
\\
&&\quad \le\mathbf P \bigl( \bigl|\xi_{x,h}(X_{t_n})\bigr|+L
h+C_{x,h,L} \epsilon _n + L \epsilon_n^\h \bigl\| B^H
\bigr\|_{0,t_n,\mathfrak{h}} \ge Lh + C_{x,h,L} \epsilon_n+ \zeta,\A^{\mathrm{\sc
{nw}}}_{t_n,h} \bigr)\qquad
\\
&&\quad\le\mathbf P \bigl( \bigl|\xi_{x,h}(X_{t_n})\bigr|\ge
\zeta/2 , \A ^{\mathrm{\sc{nw}}}_{t_n,h} \bigr) + \mathbf P \bigl( L
\epsilon_n^\h \bigl\| B^H \bigr\|_{0,t_n,\mathfrak
{h}}\ge
\zeta/2 \bigr)
\nonumber
\\
& &\quad\le2\exp \biggl( -\frac{\rho^2  (1-H)  \zeta^2}{16
\alpha_H^2} t_n^{2\beta}
\biggr)+ \mathbf P \biggl( \bigl\| B^H \bigr\| _{0,t_n,\mathfrak{h}}\ge
\frac{\zeta}{2 L \epsilon_n^\h
} \biggr) .\nonumber
\end{eqnarray}
We treat the last term in the right-hand side of the above inequality.
We need a Fernique's type lemma for the exponential moment of the H\"
{o}lder norm of the trajectories of the fBm $B^H$.
Such a result is stated in Lemma \ref{fernique} in the Appendix \ref{app1}.
Chebyshev's exponential inequality yields
%
%e31 #&#
\begin{eqnarray}
\label{yu2} \mathbf P \biggl( \bigl\| B^H \bigr\|_{0,t_n,\mathfrak{h}}\ge
\frac
{\zeta}{2 L \epsilon_n^\h} \biggr) & \le&\exp \biggl( -\frac{\zeta}{2 L \epsilon_n^\h} \biggr) \E \bigl(
\exp \bigl(\bigl\| B^H \bigr\|_{0,t_n,\mathfrak{h}} \bigr) \bigr)
\nonumber
\\[-8pt]
\\[-8pt]
\nonumber
& \le& c_{H,\mathfrak{h}} \bigl(1+t_n^{H-\mathfrak{h}} \bigr) \exp
\biggl( \frac{128 H^2}{\mathfrak{h}^2} t_n^{2(H-\mathfrak{h})} -\frac{\zeta}{2 L \epsilon_n^\h}
\biggr) ,
\end{eqnarray}
where we have used \eqref{esti-exp} from Lemma \ref{fernique}.
We recall that $ \epsilon_n\asymp n^{-(\gamma-1)/\gamma}$ where $n$
is the number of approximation points satisfying
$n=t_n^\gamma$ with $\gamma>0$.
We may write \eqref{yu2} as
\begin{eqnarray*}
&&\mathbf P \biggl( \bigl\| B^H \bigr\|_{0,t_n,\mathfrak{h}}\ge\frac
{\zeta}{2 L \epsilon_n^\h}
\biggr)\\
&&\quad \le c_{H,\mathfrak{h}} \bigl(1+t_n^{H-\mathfrak{h}} \bigr)\exp \biggl( -\frac{\zeta}{2 L} t_n^{\h
(\gamma-1)} \biggl(
1 - \frac{256 H^2  L}{\zeta \mathfrak{h}^2} t_n^{2(H-\mathfrak{h})} t_n^{\h(1-\gamma)}
\biggr) \biggr). %\\
%&
% \le c_{H,\mathfrak{h}}  \big(1+t_n^{H-\mathfrak{h}}\big)  \exp
% \left( 1 -
% C   t_n^{2H-\mathfrak{h}-\gamma\h}\right)\right)
\end{eqnarray*}
If we choose $\h$ such that $H>\h> 2H/(\gamma+1)$, then
$2H-\mathfrak{h}-\gamma\h<0$.
For simplicity, we fix
\[
\h_0 =\frac{H}2 + \frac{H}{\gamma+1} = \frac{H(\gamma
+3)}{2(\gamma+1)}
.
\]
When $t_n$ is large, more precisely:
%
%e32 #&#
\begin{equation}
\label{tau0} t_n\ge\tau_0:= \biggl( \frac{4608 L}{\zeta}
\biggr)^{{2}/{(H(\gamma-1))}}\vee1 ,
\end{equation}
we have
\[
\biggl( 1 - \frac{256 H^2  L}{\zeta \mathfrak{h}_0^2} t_n^{2(H-\mathfrak{h}_0)}
t_n^{\h_0(1-\gamma)} \biggr) \ge\frac
{1}2
\]
and \eqref{yu2} yields
%
%e33 #&#
\begin{equation}
\label{yu3} \mathbf P \biggl( \bigl\| B^H \bigr\|_{0,t_n,\mathfrak{h}_0}\ge
\frac
{\zeta}{2 L \epsilon_n^{\h_0}} \biggr)  \le2 c_{H,\mathfrak{h}_0} t_n^{H-\mathfrak{h}_0}
\exp \biggl( -\frac{\zeta}{4 L} t_n^{\h_0(\gamma-1)} \biggr) .
\end{equation}
With $\gamma>1$ we have
\[
\h_0(\gamma-1)= \frac{H(\gamma-1)(\gamma+3)}{2(\gamma+1)}\ge \frac{2H(\gamma-1)}{\gamma+1} .
\]
We report \eqref{yu3} in \eqref{yu} and we deduce \eqref{eqdidi}.

The proof of Theorem \ref{th-dev} is now complete.
\end{pf*}

In the end, we get rid of the conditional result in the next subsection.
%s3.3 #&#
\subsection{Consistency of the Nadaraya--Watson estimators}
We start the investigation of the consistency of our estimators by the
following proposition that
is anecdotal but interesting in itself.

The strong consistency of $\hat b^{\mathrm{\sc{nw}}}_{t,h}(x)$ is
naturally related to the almost-sure convergence to $0$
of $\xi_{x,h}(X_t)$ (see the decomposition \eqref{decomp-nw}). If
$H=1/2$, by the strong law of large numbers for martingales,
this convergence will holds as soon as $\int_0^\infty N^2  ( \frac
{X_{t_k} -x}h  )\,\d s =+\infty$ almost-surely.
When $H<1/2$, such a condition will also ensure us the convergence of
the fractional
stochastic term $\xi_{x,h}(X_t)$.
%
%pr4 #&#
\begin{prop}\label{topp}
Under the Hypothesis \textup{\ref{hypb}(a)}, when $H<1/2$ and
%
%e34 #&#
\begin{equation}
\label{fff} \int_0^{\infty}N^2 \biggl(
\frac{X_s -x}h \biggr) \,\d s=+\infty\qquad \mbox{$\mathbf{P}$-a.s.}
\end{equation}
then the Nadaraya--Watson estimator is strongly consistent:
%
%e35 #&#
\begin{equation}
\label{s1} \hat b^{\mathrm{\sc{nw}}}_{t,h}(x) \mathop{{\hbox to 50pt{
\rightarrowfill}}}^{\mathbf{P}\mbox{-}a.s.}_{t\to\infty,h\to 0} b(x) .
\end{equation}
\end{prop}
Of course when the bandwidth is time dependent with $\lim_{t\to\infty
}h_t=0$, we have
$\lim_{t\to\infty} \hat b^{\mathrm{\sc{nw}}}_{t,h_t}(x)=b(x)$
almost-surely.
The proof Proposition \ref{topp} is based on the following fractional
version of the integral Toeplitz lemma.
%
%le5 #&#
\begin{lemme}\label{toeplitz}
Let ${\alpha}>0$. Let $(x_t)_{t\ge0}$ be a continuous real function
such that $\lim_{t\to\infty}x_t = x$ and let $(\gamma_t)_{t\ge0}$
be a measurable, positive and bounded function. Then it holds that
\[
\frac{\int_0^t (t-s)^{\alpha-1}  ( \int_0^s\gamma_r \,\d r
) x_s \,\d s }{\int_0^t (t-s)^{\alpha-1}  ( \int_0^s\gamma_r \,\d r
) \,\d s } \mathop{\longrightarrow}_{t\to\infty}x,
\]
provided that $\lim_{t\to\infty} \int_0^t \gamma_s\,\d s =+\infty$.
\end{lemme}
\begin{pf}
Let $\eps>0$ and $A$ be such that $|x_s-x|<\eps$ for $s>A$. We denote
$C_A=\sup_{s\le A}|x_s-x|$.
By Fubini's theorem
\[
\int_0^t (t-s)^\alpha
\gamma_s \,\d s =\alpha\int_0^t
(t-s)^{\alpha
-1} \biggl( \int_0^s
\gamma_r \,\d r \biggr) \,\d s ,
\]
and we write for $t>A$
%
%e36 #&#
\begin{eqnarray}
\label{ioio} \biggl\llvert \frac{\int_0^t (t-s)^{\alpha-1}  ( \int_0^s\gamma_r \,\d r
) x_s \,\d s }{\int_0^t (t-s)^{\alpha-1}  ( \int_0^s\gamma
_r \,\d r  ) \,\d s } - x \biggr\rrvert & \le&
\frac{\int_0^t (t-s)^{\alpha
-1}  ( \int_0^s\gamma_r \,\d r  ) | x_s-x | \,\d s }{\int_0^t
(t-s)^{\alpha-1}  ( \int_0^s\gamma_r \,\d r  ) \,\d s }
\nonumber
\\[-8pt]
\\[-8pt]
\nonumber
& \le&\eps+ C_A \frac{\int_0^A (t-s)^{\alpha-1}  ( \int_0^s\gamma_r \,\d r  ) \,\d s }{\int_0^t (t-s)^{\alpha-1}  ( \int_0^s\gamma_r \,\d r  ) \,\d s } .
\end{eqnarray}
Another application of Fubini's theorem implies that
\begin{eqnarray*}
\frac{\int_0^A  ( \int_r^A(t-s)^{\alpha-1} \,\d s  ) \gamma
_r\,\d r }{\int_0^t  ( \int_r^t(t-s)^{\alpha-1} \,\d s  ) \gamma
_r\,\d r } & =& \frac{\int_0^A \gamma_r  [ (t-r)^{\alpha
}-(t-A)^{\alpha} ] \,\d r }{\int_0^t (t-r)^{\alpha}\gamma_r\,\d r}
\\
& \le& \frac{\int_0^A (t-r)^{\alpha}\gamma_r \,\d r }{\int_0^t
(t-r)^{\alpha}\gamma_r\,\d r}
\\
& \le& \frac{A t^{\alpha} ( \sup_{s\ge0}|\gamma_s|  )
}{\int_0^{t/2} (t-r)^{\alpha}\gamma_r \,\d r }
\\
& \le& \frac{A t^{\alpha} ( \sup_{s\ge0}|\gamma_s|
)}{t^{\alpha}\int_0^{t/2} \gamma_r \,\d r }
\end{eqnarray*}
and the last term tends to $0$ as $t\to\infty$. We report this
convergence in \eqref{ioio} and we obtain the result.
\end{pf}
Now we prove \eqref{s1}.
\begin{pf*}{Proof of equation (\ref{s1})}
By \eqref{decomp-nw} and \eqref{rh}, we have
\[
\bigl| \hat b^{\mathrm{\sc{nw}}}_{{t,h}}(x) -b(x)\bigr | \le L_x
h^\bb+ \bigl|\xi_{x,h}(X_t) \bigr| .
\]
Let ${\alpha}=1/2-H>0$. By the stochastic Fubini theorem,
\begin{eqnarray*}
&&\int_0^t(t-s)^{\alpha}N \biggl(
\frac{X_s -x}h \biggr)\,\d B_s\\
&&\quad ={\alpha}\int_0^t
(t-s)^{{\alpha}-1} \biggl( \int_0^s N
\biggl( \frac{X_r -x}h \biggr) \,\d B_r \biggr) \,\d s
\end{eqnarray*}
and we write $\xi_{x,h}(X_t) = T^1_{x,h}(X_t)\times T^2_{x,h}(X_t)$ with
\begin{eqnarray*}
T^1_{x,h}(X_t)  &=& \biggl(\alpha\int_0^t (t-s)^{{\alpha}-1} \biggl(
\int_0^s N^2  ( {(X_r -x)}/h  ) \,\d r \biggr)
\biggl({\biggl(\int_0^s N  ( {(X_r -x)}/h  ) \,\d B_r\biggr)}\\
&&\hspace*{185pt}{}\Big/{\biggl(\int_0^s N^2  ( {(X_r -x)}/h  ) \,\d r\biggr)}\biggr) \,\d s\biggr) \\
&&{}\Big/{\biggl(\int_0^t
\alpha_H (t-s)^{{\alpha}-1} \biggl( \int_0^s N^2  (
{(X_r -x)}/h  )\,\d r \biggr) \,\d s\biggr)}
\end{eqnarray*}
and
\begin{eqnarray*}
T^2_{x,h}(X_t)  = \frac{\int_0^t (t-s)^{\alpha}N^2  (
{(X_s -x)}/h  ) \,\d s}{\int_0^t (t-s)^{2{\alpha}} N  ( {(X_s -x)}/h  ) \,\d s } .
\end{eqnarray*}
Since $\int_0^{\infty} N^2  ( \frac{X_s -x}h
) \,\d s = +\infty$ almost-surely,
\[
\frac{\int_0^s N   ( {(X_r -x)}/h  ) \,\d B_r}{\int_0^s N^2   ( {(X_r -x)}/h  ) \,\d r} \mathop{\longrightarrow}^{\mathrm{a.s.}}_{t\to\infty}0
\]
and the generalized Toeplitz Lemma \ref{toeplitz} yields that $\lim_{t\to\infty} T^1_{x,h}(X_t) =0$ almost-surely.
Now we prove that the second term $T^2_{x,h}(X_t)$ is bounded when $t$
is large.
Since the kernel function $N$ satisfies $0\le N^2 \le N$, we have
\begin{eqnarray*}
\bigl\llvert T^2_{x,h}(X_t) \bigr\rrvert & \le&
\frac{\int_0^t (t-s)^{\alpha
}N  ( {(X_s -x)}/h  ) \,\d s}{\int_0^t
(t-s)^{2{\alpha}} N   ( {(X_s -x)}/h  ) \,\d s }
\\
& \le&\frac{ (\int_0^t (t-s)^{2{\alpha}} N   (
{(X_s -x)}/h  ) \,\d s )^{1/2} (\int_0^t N  ( {(X_s -x)}/h  ) \,\d s )^{1/2}}{\int_0^t
(t-s)^{2{\alpha}} N   ( {(X_s -x)}/h  ) \,\d s },
\end{eqnarray*}
where we have used the Cauchy--Schwarz inequality. For $t$ big enough in
such a way that
$\int_0^{t-1} N   ( \frac{X_s -x}h  ) \,\d s \ge2$,
we may write
\begin{eqnarray*}
\bigl\llvert T^2_{x,h}(X_t) \bigr\rrvert
^2 & \le&\frac{\int_0^t N   ( {(X_s -x)}/h
) \,\d s}{\int_0^{t-1} (t-s)^{2{\alpha}} N   ( {(X_s
-x)}/h  ) \,\d s }
\\
& \le&\frac{\int_0^{t-1} N   ( {(X_s -x)}/h
) \,\d s+\int_{t-1}^t N   ( {(X_s -x)}/h
) \,\d s}{\int_0^{t-1} N   ( {(X_s -x)}/h  ) \,\d s }
\\
& \le& 1 +\frac{1}2 .
\end{eqnarray*}
Therefore, $\lim_{t\to\infty}\xi_{x,h}(X_t)=0$ almost-surely and
the proof is completed.
\end{pf*}
%
%re4 #&#
\begin{remi}Let $M^{(\alpha)}=(M_t^{(\alpha)})_{t\ge0}$ with $\alpha
=1/2-H>0$ be the fractional martingale
(as so called in \cite{hns})
defined by $ M_t^{({\alpha})} = \int_0^t (t-s)^{\alpha} N ( \frac
{X_s -x}h ) \,\d B_s$.
We have seen in the previous proof that \eqref{fff} insures us that
the fractional martingale
$M^{({\alpha})}$ satisfies the strong law of large numbers:
\[
\frac{M_t^{({\alpha})}}{ \prec M^{({\alpha})}\succ_t} \mathop{{\hbox to 25pt{\rightarrowfill}}}^{\mathbf{ P}\mbox{-}\mathrm{a.s.}}_{t\to\infty}0
.
\]
with a ``fractional bracket'' defined by $ \prec M^{({\alpha})}\succ_t
= \int_0^t (t-s)^{2{\alpha}} N^2 ( \frac{X_s -x}h ) \,\d s$.
This is to our knowledge the first result of asymptotic behaviour for
fractional martingales. Unfortunately the technics we employed
to prove this convergence are not adapted to prove a similar result for
a fractional martingale with ${\alpha}<0$. See also \cite{s-expo} for
further discussions on this topic.
\end{remi}

Now we will work under the one-sided dissipative Lipschitz condition
that ensures the ergodic properties of the observation process $X$.
Before stating the main result of this paper, we make the following
non-degeneracy assumption of the stationary solution.
%
%hy4 #&#
\begin{hyp}\label{reg-inv}
The law of $\bar X$ is non-degenerate in a neighbourhood of $x$: for
any small bandwidth $h$ it holds
\[
{\mathbf E} \biggl[N \biggl( \frac{\bar
X-x}{h} \biggr) \biggr] >0 .
\]
\end{hyp}
%
%re5 #&#
\begin{remi}\label{rem-reg-inv}
It seems important to understand when the law of the stationary
solution is non-degenerate. This will be certainly the subject of
future works.
It is obviously true if the distribution of $\bar X$ has full support.
Nevertheless, the above hypothesis is satisfied in the case of the
ergodic fractional Ornstein--Uhlenbeck process.
\end{remi}
%
%th6 #&#
\begin{theo}\label{as}
We assume that Hypotheses \textup{\ref{hypb}(b)} and \textup{(c)} hold true.

When the whole trajectory is observed, the Nadaraya--Watson estimator is
consistent:
%
%e37 #&#
\begin{equation}
\label{sc1} \hat b^{\mathrm{\sc{nw}}}_{t,h}(x) \mathop{{\hbox to 50pt{
\rightarrowfill}}}_{t\to\infty,h\to
0} b(x) ,\qquad \cases{\mbox{almost-surely when
$H<1/2$;}\vspace*{2pt}\cr
\mbox{in probability when $H>1/2$.} }
\end{equation}
%
%where the above convergence holds almost-surely when $H<1/2$ and in
%probability when $H>1/2$.
Its discretized version is also consistent when we assume that the
number of approximation points satisfies $n\asymp t_n^\gamma$ with
$\gamma>1+\mm^2 H$:
%
%e38 #&#
\begin{equation}
\label{sc2}  \hat b^{\mathrm{\sc{nw}}}_{t_n,h}(x) \mathop{{\hbox to 50pt{
\rightarrowfill}}}^{\mathrm{in\ probability}}_{n\to\infty,h\to
0} b(x) .
\end{equation}
\end{theo}
We observe that the number of approximation points depends on the
regularity of $b$ in the above result. This has already been
discussed in Remark \ref{talk}.
\begin{pf*}{Proof of Theorem \ref{as}}
In the following arguments, we will make use of Proposition
\ref{ergo} with
$\ffi=N$ or $\ffi= N^2$. Obviously, \eqref{hyp-ffi} is satisfied
with $\ppp=0$.

\textit{Step \textup{1:} Proof of \textup{\eqref{sc1}}}.
When $H<1/2$, by \eqref{ergo-lim1} we obtain
\[
\frac{1}t \int_0^t N^2
\biggl( \frac{X_s -x}h \biggr) \,\d s \mathop{{\hbox to 25pt{\rightarrowfill}}}^{\P\mbox{-}\mathrm{a.s.}}_{t\to\infty}
{\mathbf E} \biggl(N^2 \biggl( \frac{\bar X-x}{h} \biggr) \biggr) .
\]
By Cauchy--Schwarz inequality and Hypothesis \ref{reg-inv}
\[
{\mathbf E} \biggl[N^2 \biggl( \frac{\bar X-x}{h} \biggr) \biggr]\ge
\biggl( {\mathbf E} \biggl[N \biggl( \frac{\bar X-x}{h} \biggr) \biggr]
\biggr)^{1/2} >0 ,
\]
and \eqref{fff} is satisfied. Thereby
\eqref{sc1} is a consequence of \eqref{s1} from Proposition \ref{topp}.

Now $H>1/2$ and let $\eps>0$. We use the probability deviation bound
\eqref{poo2} with $\zeta=t^{-\beta/2}$. For $t$ large enough and
small $h$, we have
\begin{eqnarray*}
\mathbf P \bigl( \bigl| \hat b^{\mathrm{\sc{nw}}}_{t,h}(x)- b(x) \bigr|\ge \eps
\bigr) &\le&\mathbf P \bigl( \bigl| \hat b^{\mathrm{\sc{nw}}}_{t,h}(x)- b(x) \bigr|\ge
Lh^\bb+ t^{-\beta/2} \bigr)
\\
& \le&\mathbf P \bigl( \bigl| \hat b^{\mathrm{\sc{nw}}}_{t,h}(x)- b(x) \bigr|\ge
Lh^\bb+ t^{-\beta/2} , \A^{\mathrm{\sc
{nw}}}_{t,h} \bigr) +
\P \bigl( \Omega\setminus\A^{\mathrm{\sc
{nw}}}_{t,h} \bigr)
\\
&\le&2\exp \bigl( -\rho^2 (1-H) t^{\beta} \bigr) + \P \bigl(
\Omega\setminus\A^{\mathrm{\sc{nw}}}_{t,h} \bigr).
\end{eqnarray*}
The consistency \eqref{sc1} will be proved as soon as
%
%e39 #&#
\begin{eqnarray}
\label{ant} \P \bigl( \Omega\setminus\A^{\mathrm{\sc{nw}}}_{t,h} \bigr)=\P
\biggl( \int_0^t \alpha_H
(t-s)^{1-2H} N \biggl( \frac{X_s -x}h \biggr) \,\d s \le\rho
t^{1-H+\beta} \biggr) \mathop{{\hbox to 50pt{\rightarrowfill}}}_{t\to\infty,h\to0}0
.
\end{eqnarray}
Since $(t-s)^{1-2H}\ge t^{1-2H}$,
\begin{eqnarray*}
\Omega\setminus\A^{\mathrm{\sc{nw}}}_{t,h} &\subseteq& \biggl\{ \int
_0^t N \biggl( \frac{X_s -x}h \biggr) \,\d s
\le \frac{\rho}{\alpha_H} t^{H+\beta} \biggr\}\\
&\subseteq& \biggl\{
\frac{1}t \int_0^t N \biggl(
\frac{X_s -x}h \biggr) \,\d s \le \frac{\rho}{\alpha_H} t^{H+\beta-1} \biggr\} :=
\mathcal A_{t,h} .
\end{eqnarray*}
By the ergodic result \eqref{ergo-lim1},
\[
\frac{1}t \int_0^t N \biggl(
\frac{X_s -x}h \biggr) \,\d s \mathop{{\hbox to 25pt{\rightarrowfill}}}^{\P\mbox{-}\mathrm{a.s.}}_{t\to\infty}
{\mathbf E} \biggl[N \biggl( \frac{\bar X-x}{h} \biggr) \biggr] >0
\]
thus $\1_{\mathcal A_{t,h}}$ tends to $0$ almost-surely when we choose
$\beta$ such that $H+\beta-1<0$. This implies~\eqref{ant} and~\eqref
{sc1} is proved.

\textit{Step \textup{2:} Consistency under discrete observations}.
When $H>1/2$, the proof is identical to the above one. We use the
probability deviation bound \eqref{eqdidi} instead of \eqref{poo2}
and the discrete ergodic property \eqref{ergo-lim2} is invoked in
place of \eqref{ergo-lim1}.

When $H<1/2$, we use the deviation bound \eqref{eqdidi} and it remains
to prove that
\begin{eqnarray*}
\P \bigl( \Omega\setminus\A^{\mathrm{\sc{nw}}}_{t_n,h} \bigr)=\P \Biggl(
\frac{\alpha_H}n \sum_{k=0}^{n-1}
(t_n-t_k)^{1-2H} N \biggl( \frac{X_{t_k} -x}h
\biggr) \le \rho t_n^{1-H+\beta} \Biggr) \mathop{{\hbox to 50pt{
\rightarrowfill}}}_{n\to\infty,h\to0}0 .
\end{eqnarray*}
Let $m$ be such that $t_{m-1}<\frac{t_n}2\le t_m$. Since $s\mapsto
(t_n-s)^{1-2H}$ is a decreasing function,
$(t_n-t_k)^{1-2H}\ge(t_n-\frac{t_n}2)^{1-2H} = (\frac{t_n}2)^{1-2H}$
for $k\le m-1$ and we deduce
\begin{eqnarray*}
&&\int_{0}^{t_n} \sum_{k=0}^{n-1}(t_n-t_k)^{1-2H}
N  \biggl( \frac{X_{t_k} -x}h \biggr) \1_{[t_k{,}t_{k+1})}(s) \,\d s
\\
&&\quad \ge \biggl( \frac{t_n}2 \biggr)^{1-2H} \int
_{0}^{t_n} \sum_{k=0}^{m-1}
N \biggl( \frac{X_{t_k} -x}h \biggr) \1 _{[t_k{,}t_{k+1})}(s) \,\d s
\\
&&\quad \ge \biggl( \frac{t_n}2 \biggr)^{1-2H} \int
_{0}^{ {t_n}/2} \sum_{k=0}^{m-1}
N \biggl( \frac{X_{t_k} -x}h \biggr) \1_{[t_k{,}t_{k+1})}(s) \,\d s
\\
&&\quad \ge \biggl( \frac{t_n}2 \biggr)^{1-2H} \int
_{0}^{{t_n}/2} \sum_{k=0}^{m-1}
N \biggl( \frac{X_{\tilde t_k}
-x}h \biggr) \1_{[\tilde t_k{,}\tilde t_{k+1})}(s) \,\d s ,
\end{eqnarray*}
where $\tilde t_k = t_k $ for $k\le m-1$ and $\tilde t_m = t_n/2$. We
notice that the mesh size of this new time
discretization is less that $\epsilon_n$. The observation times are no
more equally spaced but it is easy to convince ourselves
that it does not affect the results of Proposition \ref{ergo}. Thereby
\begin{eqnarray*}
\Omega\setminus\A^{\mathrm{\sc{nw}}}_{t_n,h} &\subseteq& \Biggl\{ \biggl(
\frac{t_n}2 \biggr)^{1-2H} \int_{0}^{{t_n}/2}
\sum_{k=0}^{m-1} N \biggl(
\frac{X_{\tilde t_k} -x}h \biggr) \1_{[\tilde t_k{,}\tilde t_{k+1})}(s) \,\d s \le \frac{\rho
}{\alpha_H}
t_n^{1-H+\beta} \Biggr\}
\\
& \subseteq& \Biggl\{ \frac{1}{\tilde t_m} \int_{0}^{\tilde t_m}
\sum_{k=0}^{m-1} N \biggl(
\frac{X_{\tilde t_k} -x}h \biggr) \1_{[\tilde t_k{,}\tilde t_{k+1})}(s) \,\d s \le \frac{4\rho
}{\alpha_H} (\tilde
t_m)^{H+\beta-1} \Biggr\}
\end{eqnarray*}
and accordingly of the inequality \eqref{ergo-lim2} from Proposition
\ref{ergo}, we may conclude as in the first step.
\end{pf*}
%
%s4 #&#
\section{The locally linear estimate}\label{loclin}
We follow the same structure than Section \ref{nada}. Nevertheless,
since the notation become more heavy,
we distinguish the case of continuous observations from the discrete one.
%s4.1 #&#
\subsection{Heuristic approach}
The idea is similar to the one used Section \ref{nada} and also
follows the one developed in \cite{spo}.
At first we discuss the case of a linear drift coefficient $b$ of the form
$b_{\theta_0,\theta_1}(z) = \theta_0 + \theta_1(z-x)/h$. Hence, $b$
depends on two
parameters $\theta_0$ and $\theta_1$ ($h>0$ is fixed). Since
$b_{\theta_0,\theta_1}(x) = \theta_0$,
an estimator of $\theta_0$ is an estimator of the value of the drift
at the point $x$.
We denote
\[
X^{\theta}_t = x_0 +\int_0^t
b_{\theta_0,\theta_1}\bigl(X^\theta_s\bigr)\,\d s +
B_t^H.
\]
Similarly to \eqref{YM2}, we introduce the observable process
$Y^\theta=(Y^\theta_t)_{t\ge0}$ defined by
\begin{eqnarray*}
Y^\theta_t & = &x_0 + \int_0^t
\tilde w_H(t,s) b_{\theta_0,\theta
_1}\bigl(X^\theta_s
\bigr) \,\d \bigl\langle M^H\bigr\rangle_s +
M^H_t
\\
& =& x_0 + \int_0^t
\rho_s^{\mathrm{T}} \times\pmatrix{
\theta_0
\vspace*{2pt}\cr
\theta_1} \,\d \bigl\langle M^H\bigr
\rangle_s + M^H_t ,
\nonumber
\end{eqnarray*}
with $\rho=(\rho_s)_{s\ge0}$ is the process with values in $\mathbf
R^2$ defined by
\[
\rho_s = \tilde w_H(t,s)\pmatrix{ 1
\vspace*{2pt}\cr
\bigl(X^\theta_s-x\bigr)/h},
\]
and for a matrix $A$, $A^{\mathrm{T}}$ denotes its transpose.
Intuitively, the values $\theta_0$ and $\theta_1$ can
be estimated by the least squares method (see, e.g., \cite{lbm}).
If the $2\times2$-matrix
\[
\Pi_t=\int_0^t \rho_s
\rho_s^{\mathrm{T}} \,\d \bigl\langle M^H\bigr
\rangle_s
\]
is not singular,
the least squares estimator of $(\theta_0,\theta_1)^{\mathrm{T}}$
obtained at time $t$ is given by
\[
\hat\theta(t) = \pmatrix{ \hat
\theta_0(t)
\vspace*{2pt}\cr
\hat\theta_1(t)} = \Pi_t^{-1} \int
_{0}^t \rho_s \,\d Y^\theta_s
.
\]
With the constant $\alpha_H$ defined in \eqref{alfa}, we may write
\begin{eqnarray*}
\Pi_t & =& \int_0^t \bigl( \tilde
w_H (t,s) \bigr)^2 \pmatrix{ 1 & \bigl(X^\theta_s-x\bigr)/h
\vspace*{2pt}\cr
\bigl(X^\theta_s-x\bigr)/h & \bigl(X^\theta_s-x
\bigr)^2/h^2} \,\d \bigl\langle M^H\bigr
\rangle_s
\\
& =& \int_0^t
\alpha_H^2(t-s)^{1-2H} \pmatrix{ 1 & \bigl(X^\theta_s-x\bigr)/h
\vspace*{2pt}\cr
\bigl(X^\theta_s-x\bigr)/h & \bigl(X^\theta_s-x
\bigr)^2/h^2} \,\d s %
\end{eqnarray*}
and we obtain the following expression of $\hat\theta_0(t)$:
\[
\hat\theta_0(t)  = \frac{m_2(t)}{\delta(t)} \int_0^t
\tilde w_H(t,s) \,\d Y^\theta_s - \frac{m_{1}(t)}{\delta(t)}
\int_0^t \tilde w_H(t,s) \biggl(
\frac{X_s^\theta-x}h \biggr) \,\d Y^\theta_s , %\hat\theta_1(t) & = -\frac{m_1(t)}{\delta(t)} \int_0^t \tilde w_H(t,s)
%dY^\theta_s
%+ \frac{m_{0}(t)}{\delta(t)} \int_0^t \tilde w_H(t,s) \mbox{$\left(
\]
where for $i\in\{0,1,2\}$:
\[
m_i(t) = \int_0^t
\alpha_H^2(t-s)^{1-2H} \biggl( \frac
{X_s^\theta-x}h
\biggr)^i \,\d s
\]
and
\[
\delta(t) = m_0(t)m_2(t)-m_1^2(t).
\]
In the context of the fractional diffusion \eqref{eds}, the drift
$b$ is not linear. Hence, we approximate it by a linear function
$\theta_0 + \theta_1(z-x)/h$ in a neighbourhood $[x-h,x+h]$ of the
point $x$. For this purpose, we use a kernel function $N$ that
satisfies Hypothesis \ref{kernel}.
The above discussion is the starting point of the construction of a
locally linear estimator of $b$ under continuous observations.
%s4.2 #&#
\subsection{Observations based on the whole trajectory}
%s4.2.1 #&#
\subsubsection{Construction and decomposition of the error}
We give the following definition of a locally linear estimator of $b(x)$
by means of the observable processes $Y$ and $X$.
%
%de2 #&#
\begin{defi}\label{def2}
The locally linear estimators at time $t$ of $b(x)$ with the kernel $N$
and a bandwidth $h$ is defined by
%
%e40 #&#
%e41 #&#
\begin{eqnarray}
\label{Yhatb} \hat b^{\mathrm{ll}}_{{t,h}}(x) & =& \frac{\v_2(t)}{\d(t)}
\int_0^t \tilde w_H(t,s) N \biggl(
\frac{X_s -x}h \biggr) \,\d Y_s
\nonumber
\\[-8pt]
\\[-8pt]
\nonumber
&&{} - \frac{\v_{1}(t)}{\d(t)} \int
_0^t \tilde w_H(t,s) \biggl(
\frac{X_s -x}h \biggr) N \biggl( \frac{X_s -x}h \biggr) \,\d Y_s
\\
\label{defhatb} & =& \int_0^t \biggl[
\frac{\v_2(t)}{\d(t)} - \frac{\v_{1}(t)}{\d
(t)} \biggl( \frac{X_s -x}h \biggr) \biggr]
\alpha_H^2 (t-s)^{1-2H} N \biggl(
\frac{X_s -x}h \biggr) \,\d X_s ,
\end{eqnarray}
where for $j=0,1,2$:
%
%e42 #&#
\begin{equation}
\label{vv}  \cases{
\displaystyle\v_j(t)  = \int_0^t
\alpha_H^2 (t-s)^{1-2H} \biggl( \frac{X_s -x}h
\biggr)^j N \biggl( \frac
{X_s -x}h \biggr) \,\d s,
\vspace*{2pt}\cr
\d(t)  = \v_0(t) \v_2(t) - \bigl(\v_{1}(t)
\bigr)^2 . }
\end{equation}
\end{defi}
The alternative expression \eqref{defhatb} is obtained thanks to the
definition of the process $Y$ given in~\eqref{Y}
and the relation
\[
w_H(t,s) \tilde w_H(t,s) = \alpha_H^2
(t-s)^{1-2H}\qquad \mbox{for $t>s$} .
\]

Moreover, the representation \eqref{YB}, the facts that $N\le1$ and
that for all $z\in\mathbf R$, $|zN(z)|\le1$, we notice that the
stochastic integrals in \eqref{Yhatb} are well defined.
Moreover, we remark that $\d(t)>0$ by the Cauchy--Schwarz inequality.

In order to understand what kind of quantities will appear in the
deviation probability bound, we write a
decomposition of the error $ \hat b^{\mathrm{ll}}_{{t,h}}(x)-b(x)$.
Using \eqref{YB} and \eqref{Yhatb}, we rewrite $\hat b^{\mathrm
{ll}}_{t,h}(x)$ as
\begin{eqnarray*}
\hat b^{\mathrm{ll}}_{t,h}(x) & = &\frac{\v_2(t)}{\d(t)} \int
_0^t \alpha_H (t-s)^{1/2-H}
N \biggl( \frac{X_s -x}h \biggr) \,\d B_s
\\
&&{} - \frac{\v_{1}(t)}{\d(t)} \int_0^t
\alpha_H (t-s)^{1/2-H} \biggl( \frac{X_s -x}h \biggr) N
\biggl( \frac{X_s -x}h \biggr) \,\d B_s
\\
&&{} + \int_0^t \biggl[\frac{\v_2(t)}{\d(t)} -
\frac{\v_{1}(t)}{\d
(t)} \biggl( \frac{X_s -x}h \biggr) \biggr]
\alpha_H^2 (t-s)^{1-2H} N \biggl(
\frac{X_s -x}h \biggr) b(X_s) \,\d s .
\end{eqnarray*}
Now we consider the local error functions $\delta_{x,h}$ defined by
\[
\delta_{x,h}(z)  = b(z)- \bigl( b(x)+ b'(x)\times(z-x)
\bigr).
\]
By the definitions of the functions $\v_j$ it holds that
\begin{eqnarray*}
&& \int_0^t \biggl[\frac{\v_2(t)}{\d(t)} -
\frac{\v
_{1}(t)}{\d(t)} \biggl( \frac{X_s -x}h \biggr) \biggr]
\alpha_H^2 (t-s)^{1-2H} N \biggl(
\frac{X_s -x}h \biggr) b(X_s) \,\d s
\\
&&\quad = b(x) + \int_0^t \biggl[\frac{\v_2(t)}{\d(t)}
- \frac{\v_{1}(t)}{\d(t)} \biggl( \frac{X_s -x}h \biggr) \biggr]
\alpha_H^2 (t-s)^{1-2H} N \biggl(
\frac{X_s -x}h \biggr) \delta_{x,h}(X_s) \,\d s .
\end{eqnarray*}
Now for $j=0,1$ we denote
\begin{eqnarray*}
\nu_j(t,s) & = &\alpha_H (t-s)^{1/2-H} \biggl(
\frac{X_s -x}h \biggr)^j N \biggl( \frac{X_s -x}h \biggr),
\\
\tilde\nu_j(t,s) & =& \alpha_H^2
(t-s)^{1-2H} \biggl( \frac
{X_s -x}h \biggr)^j N \biggl(
\frac{X_s -x}h \biggr)
\end{eqnarray*}
and $(\mathrm{V}_t)_{t\ge0}$ is the process with values in $\mathbf
R^{2\times2}$ defined by
%
%e43 #&#
\begin{eqnarray}
\label{def-v} \mathrm{V}_t = \pmatrix{ \v_0(t) & \v_{1}(t)
\vspace*{2pt}\cr
\v_{1}(t) & \v_2(t)} .
\end{eqnarray}
Thus, we have again the expression
%
%e44 #&#
\begin{equation}
\label{decomp} \hat b^{\mathrm{ll}}_{t,h}(x) = b(x) +
\xi_{x,h}^1(X_t) + r_{x,h}^\mathrm{loc}(X_t),
\end{equation}
where $\xi_{x,h}^1(X_t)$ and $r_{x,h}^\mathrm{loc}(X_t)$ are the
first components of the following two dimensional vectors
\begin{eqnarray*}
\Xi_{x,h}(X_t) & %= \left(%
% \xi_{x,h}^1(X_t) \\
% \xi_{x,h}^2(X_t) \\
= &(\mathrm{V}_t)^{-1}
\int_0^t \pmatrix{ \nu_0(t,s)
\vspace*{2pt}\cr
\nu_1(t,s)} \,\d B_s,
\\
R_{x,h}^\mathrm{loc}(X_t) & %= \left(%
% \mathrm{r}_{x,h}^1(X_t) \\
% \mathrm{r}_{x,h}^2(X_t) \\
= &(
\mathrm{V}_t)^{-1} \int_0^t
\pmatrix{ \tilde\nu_0(t,s)
\vspace*{2pt}\cr
\tilde\nu_1(t,s)}\delta_{x,h}(X_s)\,\d s.
\end{eqnarray*}
The interpretation of $\xi_{x,h}^1(X_t)$ and $r_{x,h}^\mathrm
{loc}(X_t)$ and is the same one than in the Nadaraya--Watson procedure.
It is important to notice that when $H=1/2$ we obtain the same
decomposition of $\hat b^{\mathrm{ll}}_{t,h}(x)-b(x)$ as the one in
\cite{spo}, equation (5.3).
%s4.2.2 #&#
\subsubsection{Deviation probability and consistency}
In view of the term $r_{x,h}^\mathrm{loc}(X_t)$ in \eqref{decomp},
the accuracy of the locally linear estimate
will be expressed thanks to the quality of the approximation of $b$ by
a linear function.
Under Hypothesis \ref{hypb}(a) it is natural to
introduce in the neighbourhood $[x-h,x+h]$ of the point $x$ the quantity
\[
\Delta_{x,h} =\sup_{|z-x|\le h} \bigl| b(z)- \bigl( b(x)+
b'(x)\times(z-x) \bigr) \bigr|.
\]
In order to study the error from a probabilistic point of view (see
$\xi_{x,h}^1(X_t)$ in \eqref{decomp}),
we make the following comments.
If the kernel function $N$ satisfies $N^2=N$ and if $H=1/2$, the
process $V=(\mathrm{V}_t)_{t\ge0}$
defined in \eqref{def-v} is the quadratic variation process of the
two-dimensional martingale $M=(M_t)_{t\ge0}$ defined by
\[
M_t = \int_0^t \pmatrix{ 1
\vspace*{2pt}\cr
\displaystyle\frac{X_s -x}h} N \biggl( \frac{X_s -x}h \biggr)
\,\d B_s .
\]
If we investigate the strong consistency of our estimator we shall
use a strong law of large numbers for multivariate martingales (see
\cite{lbm,ku-so,cp,vz}). Therefore, the strong
consistency is a consequence of asymptotic properties as $t$ goes to
infinity of the eigenvalues of
the matrix $\mathrm{V}_t$. In the fractional framework, this kind of
asymptotic has not yet been studied.
Nevertheless, the eigenvalues of $\mathrm{V}_t$ play a
crucial role in the following.

Thus, we introduce for some $\rho>0$ and $\beta>0$ the random set
\[
\A^{\mathrm{ll}}_{t,h} = \bigl\{ \lambda_m(t)\ge\rho
t^{1-H+\beta
} \bigr\},
\]
where $\lambda_m(t)$ is the smallest eigenvalue of the matrix $V_t$.
The properties of
$\hat b^{\mathrm{ll}}_{t,h}(x)$ are first studied restricted to the
event $\A^{\mathrm{ll}}_{t,h}$.
The consistency is proved under Hypotheses \ref{hypb}(b) and (c) insuring the ergodicity of \eqref{eds} and under the
following non-degeneracy condition on the law of the stationary solution.
%
%hy5 #&#
\begin{hyp}\label{reg-inv-ll}
The law of $\bar X$ is (strongly) non-degenerate in a neighbourhood of
$x$: for any small bandwidth $h$ it holds:
\[
{\mathbf E} \biggl[ \biggl( \frac{\bar
X-x}{h} \biggr)^2 N \biggl(
\frac{\bar X-x}{h} \biggr) \biggr] >0 .
\]
\end{hyp}
Notice that we use the terminology ``strongly non-degenerate'' because
Hypothesis \ref{reg-inv-ll} implies Hypothesis \ref{reg-inv}.
%
%th7 #&#
\begin{theo}\label{th-dev-ll}Let $x$ be fixed.
\begin{longlist}[(ii)]
\item[(i)]
If $b$ satisfies Hypothesis \textup{\ref{hypb}(a)}, then for any $\zeta>0$, we have
%
%e45 #&#
\begin{eqnarray}
\label{poo}  \mathbf P \bigl( \bigl| \hat b^{\mathrm{ll}}_{t,h}(x)- b(x) \bigr|
\ge c_{\rho,H} \Delta_{x,h} t^{1-H-\beta} + \zeta , \A
^{\mathrm{ll}}_{t,h} \bigr) \le 4\exp \biggl( - \frac{(1-H)\rho^2}{2\alpha_H^2}
\zeta^2 t^{2\beta} \biggr),
\end{eqnarray}
with $c_{\rho,H}=\sqrt{2} c_H^2 /(\rho\lambda_H )$.
\item[(ii)]
When we assume Hypotheses \textup{\ref{hypb}(b)} and \textup{(c)} and Hypothesis~\ref
{reg-inv-ll}, the fractional diffusion is then ergodic and we have the
consistency of the locally linear estimator:
\[
 \hat b^{\mathrm{ll}}_{t,h}(x) \mathop{{\hbox to 50pt{
\rightarrowfill}}}^{\mathrm{in\ probability}}_{t\to\infty,h\to
0} b(x) .
\]
\end{longlist}
\end{theo}
The proof of this result is postponed in Section \ref{proofs}.

If we consider that $\zeta=t^{-\beta/2}$, \eqref{poo} implies that
\begin{eqnarray*}
\mathbf P \bigl( \bigl| \hat b^{\mathrm{ll}}_{t,h}(x)- b(x) \bigr|\ge
c_{\rho,H} \Delta_{x,h} t^{1-H-\beta} + t^{-\beta/2} ,
\A^{\mathrm{ll}}_{t,h} \bigr) \le4\exp \biggl( -\frac{(1-H)\rho^2}{2c_H^2}
t^{\beta} \biggr).
\end{eqnarray*}
The quality of the approximation of $b$ by a linear function is
measured by $\Delta_{x,h}$ in a vicinity of~$x$.
Under Hypothesis \ref{hypb}(a), we have $\Delta_{x,h}\le2L_xh$.
Assume also that $b$ is twice differentiable in a neighbourhood of $x$
with second derivative bounded by $L_x$,
then $\Delta_{x,h} \le L_xh^2/2$. Now we are able to choose a time
dependent bandwidth $h_t$.
Clearly if $h_t^2 \asymp L_x^{-1} t^{H-1+\beta/2}$ (remind that the
symbol $\asymp$ means that the
ratio of the functions are bounded), we obtain that the rate of
estimation is of
order $t^{-\beta/2}$:
\begin{eqnarray*}
 \mathbf P \bigl( \bigl| \hat b^{\mathrm{ll}}_{t,h}(x)- b(x) \bigr|\ge\bar
c_{\rho,H} t^{-\beta/2} , \A^{\mathrm{ll}}_{t,h} \bigr)
\le4\exp \biggl( -\frac{(1-H)\rho^2}{2c_H^2} t^{\beta} \biggr),
\end{eqnarray*}
where $\bar c_{\rho,H}=c_{\rho,H}/2 +1$ and of course $\beta$
has been chosen such that $\beta<2(H-1)$.

Since $b$ is unknown, we have no reason to have information about
$L_x$. As usual in such a
non-parametric context, we have two choices. On the one hand, we can
restrict our problem to a class of drift
function $b$ satisfying the above hypotheses with constants $L_x$ such
that $L_x\in(L_{\mathrm{min}}{,}L_{\mathrm{max}})$.
On the other hand, an adaptive (data-driven) choice of the bandwidth
may be considered (see \cite{spo} and the references therein).
Unfortunately, the analysis of the error given in Theorem \ref{th-dev}
seems to be not adapted
to this powerful method of bandwidth's choice.

To continue the comparison with the work of Spokoiny (see \cite{spo}),
we may relate our random set
and the one that appears in \cite{spo}. Indeed, very simple
calculations allow us to write
the exact expression of the smallest eigenvalue of the matrix $\mathrm
{V}_t$ as
%
%e46 #&#
\begin{equation}
\label{vp} \lambda_{m}(t) = \tfrac{1}2 \bigl(
\v_0(t)+ \v_2(t) - \bigl( \bigl( \v _0(t)+
\v_2(t)\bigr)^2 -4 \d(t) \bigr)^{1/2} \bigr) .
\end{equation}
The above expression employs analogous quantities that the one
appearing in the random set $\mathcal A_h$ (see \cite{spo}, page 819).
Despite these analogies, it is not easy to compare the two events.
Moreover, the discussion about the accuracy of the approximation and
the ``stochastic error'' is different from the one
made in \cite{spo}. This is due to the fact that the stochastic error
is hidden in the random set $\A^{\mathrm{ll}}_{t,h}$ whereas
it appears explicitly as a ``conditional variance'' in the work of Spokoiny.

Now we give an effective way to estimate the drift when we consider
discrete observations.

%s4.3 #&#
\subsection{Discrete observations}
We consider the discretization of the quantities that are defined in
\eqref{vv}. For $j=0,1,2$,
\begin{eqnarray*}
 \cases{%
\displaystyle \v_j^n(t_n)
 = \int_0^{t_n} \sum
_{k=1}^{n-1}\alpha_H^2
(t_n-t_k)^{1-2H} \biggl( \frac{X_{t_k}
-x}h
\biggr)^j N \biggl( \frac{X_{t_k} -x}h \biggr) \1 _{[t_k,t_{k+1})}(s)\,\d
s,
\vspace*{2pt}\cr
\d^n(t_n)  = \v^n_0(t_n)
\v^n_2(t_n) - \bigl(\v^n_{1}(t_n)
\bigr)^2 }
\end{eqnarray*}
and we denote $V_{t_n}^n$ the matrix
\[
\mathrm{V}^n_{t_n}=\pmatrix{ \v^n_0(t_n) & \v^n_{1}(t_n)
\vspace*{2pt}\cr
\v^n_{1}(t_n) & \v^n_2(t_n)}.
\]
Considering \eqref{defhatb}, we propose the following estimator of
$b(x)$ based on discrete observations.
%
%de3 #&#
\begin{defi}\label{hatn}
The discretized locally linear estimator at time $t_n$ with a bandwidth
$h$ is
\begin{eqnarray*}
\hat b^{\mathrm{ll}}_{t_n,h}(x)  = \sum_{k=1}^{n-1}
\biggl[\frac{\v
^n_2(t_n)}{\d^n(t_n)} - \frac{\v^n_{1}(t_n)}{\d^n(t_n)} \biggl( \frac{X_{t_k} -x}h \biggr)
\biggr] \alpha_H^2 (t_n-t_k)^{1-2H}
N \biggl( \frac{X_{t_k} -x}h \biggr) (X_{t_{k+1}}-X_{t_k}) .
\end{eqnarray*}
\end{defi}
As in \eqref{decomp-n}, We will decompose $\hat b^{\mathrm
{ll}}_{t_n,h}(x)$ into a sum of three terms:
%
%e47 #&#
\begin{equation}
\label{decompn} %\label{defhatb2}
\hat b^{\mathrm{ll}}_{t_n,h}(x)-b(x) =
\xi_{x,h}(X_{t_n}) + r_{x,h}^{\mathrm{loc}}(X_{t_n})+
r_{x,h}^{\mathrm{traj}}(X_{t_n}) .
\end{equation}
For this purpose, we consider the simple process
\[
T_s^n=\sum_{k=1}^{n-1}
\biggl[\frac{\v^n_2(t_n)}{\d^n(t_n)} - \frac
{\v^n_{1}(t_n)}{\d^n(t_n)} \biggl( \frac{X_{t_k} -x}h \biggr)
\biggr] \alpha_H^2 (t_n-t_k)^{1-2H}
N \biggl( \frac{X_{t_k} -x}h \biggr)\1 _{[t_k,t_{k+1})}(s) .
\]
Since $\hat b^{\mathrm{ll}}_{t_n,h}(x) = \int_0^{t_n} T^n_s\,\d X_s$, we
use \eqref{Y} and \eqref{YB} to write that
\begin{eqnarray*}
\hat b^{\mathrm{ll}}_{t_n,h}(x) & =& \int
_0^{t_n} \frac
{T^n_s}{w_H(t_n,s)} \,\d Y_s
\\
& =& \int_0^{t_n} b(X_s)
T^n_s \,\d s+\bigl(\lambda_H(2-2H)
\bigr)^{1/2} \int_0^{t_n}
\frac{T^n_s}{w_H(t_n,s)} s^{1/2-H} \,\d B_s
\\
& =& \int_0^{t_n} b(X_s)
T^n_s \,\d s+ {\alpha}_H^{-1}\int
_0^{t_n} T^n_s
(t_n-s)^{H-1/2} \,\d B_s .
\end{eqnarray*}
As in the case of continuous observations,
\[
\xi_{x,h}(X_{t_n}) = {\alpha}_H^{-1}
\int_0^{t_n} T^n_s
(t_n-s)^{H-1/2} \,\d B_s
\]
is the first component of the two dimensional random vector
\[
\Xi^n_{x,h} = \bigl(\mathrm{V}^n_{t_n}
\bigr)^{-1} \int_0^{t_n} \pmatrix{ \mu^0(t_n,s)
\vspace*{2pt}\cr
\mu^1(t_n,s)} \,\d B_s ,
\]
where for $j=0,1$:
\[
\mu^j(t_n,s)  = (t_n-s)^{H-1/2}
\sum_{k=1}^{n-1}\alpha_H
(t_n-t_k)^{1-2H} \biggl( \frac{X_{t_k} -x}h
\biggr)^j N \biggl( \frac{X_{t_k} -x}h \biggr)\1_{[t_k,t_{k+1})}(s) .
\]
The last two terms in \eqref{decompn} come from the equality:
\begin{eqnarray*}
\int_0^{t_n} T^n_s
b(X_s) \,\d s & = &\int_0^{t_n}
T^n_s \bigl( b(X_s)- b(X_{t_k})
\bigr) \,\d s + \int_0^{t_n} T^n_s
\bigl(b(X_{t_k})-b(x) \bigr) \,\d s
\\
&:=& r_{x,h}^{\mathrm{traj}}(X_{t_n})+r_{x,h}^{\mathrm{loc}}(X_{t_n}).
\end{eqnarray*}
Easy but tedious computations yield that $ r_{x,h}^{\mathrm
{traj}}(X_{t_n})$, respectively, $r_{x,h}^{\mathrm{loc}}(X_{t_n})$, is
the first component of
\[
R^{\mathrm{traj}}_{x,h}(X_{t_n})=\bigl(\mathrm{V}^n_{t_n}
\bigr)^{-1} \int_0^{t_n}\pmatrix{ \varpi^0_{\mathrm{traj}}(t_n,s)
\vspace*{2pt}\cr
\varpi^1_{\mathrm{traj}}(t_n,s)} \,\d s ,
\]
respectively,
\[
R^{\mathrm{loc}}_{x,h}(X_{t_n})=\bigl(\mathrm{V}^n_{t_n}
\bigr)^{-1} \int_0^{t_n} \pmatrix{ \varpi^0_{\mathrm{loc}}(t_n,s)
\vspace*{2pt}\cr
\varpi^1_{\mathrm{loc}}(t_n,s)} \,\d s ,
\]
where we have denoted for $j=0,1$:
\begin{eqnarray*}
\varpi^j_{\mathrm{traj}}(t_n,s) &=& \sum
_{k=1}^{n-1}\alpha_H (t_n-t_k)^{1-2H}
\biggl( \frac
{X_{t_k} -x}h \biggr)^j N \biggl( \frac{X_{t_k} -x}h
\biggr) \bigl(b(X_s)-b(X_{t_k}) \bigr)\1_{[t_k,t_{k+1})}(s),
\\
\varpi^j_{\mathrm{loc}}(t_n,s) &=& \sum
_{k=1}^{n-1}\alpha_H (t_n-t_k)^{1-2H}
\biggl( \frac
{X_{t_k} -x}h \biggr)^j N \biggl( \frac{X_{t_k} -x}h
\biggr) \delta_{x,h}(X_{t_k})\1_{[t_k,t_{k+1})}(s) .
\end{eqnarray*}

%re6 #&#
\begin{remi}
As it has been already noticed in Nadaraya--Watson's type estimation
procedure (see Section \ref{nada11}), three terms
appear in the decomposition \eqref{decompn}. Each of them have a
precise meaning:
\begin{itemize}
\item the first term, $\xi_{x,h}(X_t)$, is a ``stochastic error term'';
\item$r_{x,h}^{\mathrm{loc}}(X_{t_k})$) represents again the accuracy
of the
local approximation of $b$ by a constant function in a neighbourhood of
the point $x$ in the discrete times
$(t_k)_{0\le k\le n}$;
\item$r_{x,h}^{\mathrm{traj}}(X_{t_n})$ is the error due to the
discretization of the continuous process $(X_s)_{s\ge0}$.
\end{itemize}
\end{remi}
The next result establishes the probability deviation bound for the
discrete locally linear estimator of $b(x)$.
%
%th8 #&#
\begin{theo}\label{didill}
We assume that $b$ is Lipschitz with Lipschitz's constant $L$. There
exists $\mathfrak{u}_1,\mathfrak{u}_2>0$ such that for any $t_n$
large enough, we have the conditional deviation probability bound:
%
%e48 #&#
\begin{eqnarray}
\label{eqll} &&\mathbf P \bigl(\bigl | \hat b^{\mathrm{ll}}_{t_n,h}(x)- b(x)\bigr|
\ge c_{\rho,H} \Delta_{x,h} t_n^{1-H-\beta}+c_1
\epsilon_n t_n^{1-H-\beta} + \zeta ,
\A^{\mathrm{\sc{nw}}}_{t_n,h} \bigr)
\nonumber
\\[-8pt]
\\[-8pt]
\nonumber
&&\quad \le4\exp \biggl( -\frac{\rho^2  (1-H)  \zeta^2}{16
\alpha_H^2} t_n^{2\beta} \biggr) +
c_{H} t_n^{\mathfrak{u}_1} \exp \biggl( -\frac{\zeta
}{4 c_2}
t_n^{\mathfrak{u}_2} \biggr) ,
\end{eqnarray}
where we have set:
\begin{itemize}[$-$]
\item[$-$] $\mathcal A^{\mathrm{ll}}_{t_n,h}= \{\lambda
_m^n(t_n)\ge\rho t_n^{1-H+\beta}  \}$ with $\beta$ a positive
real number such that $1-\beta<\gamma H$;
\item[$-$] $\lambda_m^n(t_n)$ denotes
the smallest eigenvalue of the matrix $\mathrm{V}^n_{t_n}$;
\item[$-$] $c_1>0$ depends on $\rho$, $L$, $x$, $h$ and $H$;
\item[$-$] $c_2>0$ depends on $\rho$, $H$ and $L$.
\end{itemize}
\end{theo}
The proof of this result is done in Section \ref{proofs}.

All the constants in \eqref{eqll} are known explicitly. The interested
reader shall find them in the proof.
The next theorem is one of the most important result of this work since
it sets convergence of the discretized estimator toward
the unknown value $b(x)$.
%
%th9 #&#
\begin{theo}\label{consist-ll}
Assume that Hypotheses \textup{\ref{hypb}(b)}, \textup{(c)} and Hypothesis~\ref
{reg-inv-ll} hold.
If moreover the number of approximation points satisfies $n\asymp
t_n^\gamma$ with $\gamma> \max ( 1+(\mm^2+2)  H;3)$,
then the locally linear estimator of $b(x)$ is consistent:
\[
 \hat b^{\mathrm{ll}}_{t_n,h}(x) \mathop{{\hbox to 50pt{
\rightarrowfill}}}^{\mathrm{in\ probability}}_{t_n\to\infty,h\to
0} b(x) .
\]
\end{theo}
\begin{pf}
Since we apply Proposition \ref{ergo} with $z\mapsto z^2 N(z)$ that
satisfies \eqref{hyp-ffi} with $\ppp=2$, the condition
on $\gamma$ is justified. Then we just have to use \eqref{eqll} with
$1-\gamma H<\beta< 1-H$ and we argue as in the proof of \eqref{sc2}
form Theorem \ref{as}.
\end{pf}
%
%%---------------------------------------------------------------------------------------------
%%---------------------------------------------------------------------------------------------
%%---------------------------------------------------------------------------------------------
%%---------------------------------------------------------------------------------------------
%s5 #&#
\section{Proofs}\label{proofs}
%s5.1 #&#
\subsection{\texorpdfstring{Proof of Theorem \protect\ref{th-dev-ll}}
{Proof of Theorem 7}}
We split its proof into separate steps. The starting point is the
decomposition \eqref{decomp}. We treat each term of $ \hat b^{\mathrm
{ll}}_{t_n,h}(x)- b(x)$ separately.

We recall some basic facts from linear algebra. We denote for a
vector ${\mathrm z}=(z_1,z_2)^{\mathrm{T}}\in\mathbf R^2$, $\|
{\mathrm z} \|_\infty= |z_1|\vee|z_2|$ and $\|{\mathrm z}\|_2 $ its
Euclidian norm. For any $t\ge0$, $0<\lambda_m(t)\le\lambda_M(t)$
are the eigenvalues of the symmetric matrix $\mathrm{V}_t$. For $
{\mathrm y}=(y_1{,}y_2)^{\mathrm{T}}\in\mathbf R^2$, we denote $
{\mathrm z}=(z_1,z_2)^{\mathrm{T}} = (\mathrm{V}_t)^{-1} {\mathrm y}
$ and it holds
%
%e49 #&#
\begin{equation}
\label{norm} \|{\mathrm z}\|_\infty \le \| {\mathrm z}\|_2
= \biggl( \frac{|y_1|^2}{\lambda_m(t)^2}+\frac{|y_2|^2}{\lambda
_M(t)^2} \biggr)^{1/2} \le
\sqrt{2} \biggl( \frac{|y_1|}{\lambda_m(t)}\vee\frac
{|y_2|}{\lambda_M(t)} \biggr) .%
\end{equation}
%
%s5.1.1 #&#
\subsubsection{\texorpdfstring{Proof of Theorem~\textup{\protect\ref{th-dev-ll}(i)}}
{Proof of Theorem 7(i)}}
We study $r_{x,h}^\mathrm{loc}(X_t)$. Since for any real $z$, $0\le
N(z)\le1 $ and $z^2N(z)\le N(z)$, we have the inequality $\v_2(t)\le
\v_0(t)$.
By the Cauchy--Schwarz inequality, we obtain
\begin{eqnarray*}
\biggl\llvert \int_0^t \alpha_H^2
(t-s)^{1-2H} \biggl( \frac{X_s
-x}h \biggr)N \biggl(
\frac{X_s -x}h \biggr) \delta _{x,h}(X_s) \,\d s\biggr
\rrvert %
\le \Delta_{x,h} \bigl( \v_0(t)
\v_2(t) \bigr)^{1/2} \le \Delta_{x,h}
\v_0(t) ,
\end{eqnarray*}
and thus
%
%e50 #&#
\begin{equation}
\label{nunu} \biggl\llvert \int_0^t \tilde
\nu_1(t,s) \delta_{x,h}(X_s)\,\d s\biggr\rrvert \le
\Delta_{x,h} \v_0(t) .
\end{equation}
The relation \eqref{norm} yields
\begin{eqnarray*}
\bigl|r_{x,h}^\mathrm{loc}(X_t) \bigr| & \le&\bigl\|
R_{x,h}(X_t) \bigr\|_\infty
\\
& \le&\sqrt{2} \biggl( \frac{|\int_0^t \tilde\nu_0(t,s)
\delta_{x,h}(X_s)\,\d s | }{\lambda_m(t)} \vee\frac{ |\int_0^t \tilde\nu_1(t,s) \delta_{x,h}(X_s)\,\d s| }{\lambda_M(t)}
\biggr)
\\
& \le& \sqrt{2} \Delta_{x,h} \biggl( \frac{\v_0(t)}{\lambda
_m(t)} \vee
\frac{\v_0(t)}{\lambda_M(t)} \biggr)
\end{eqnarray*}
and consequently
\[
\bigl|r_{x,h}^\mathrm{loc}(X_t) \bigr| \le \sqrt{2}
\Delta_{x,h} \frac{\v
_0(t)}{\lambda_m(t)} .
\]
Since $\v_0(t) \le( c_H^2/\lambda_H ) t^{2-2H}$, we deduce the
following bound on the random set $\mathcal A^{\mathrm{ll}}_{t,h}$
%
%e51 #&#
\begin{equation}
\label{bound-r} \bigl|r_{x,h}^\mathrm{loc}(X_t) \bigr| \le
c_{\rho,H} \Delta_{x,h} t^{1-H-\beta}
\end{equation}
with $c_{\rho,H}=\sqrt{2} c_H^2 /(\rho\lambda_H )$.

For the analysis of $\xi_{x,h}^1(X_t)$, we consider $ \zeta>0$ and by
\eqref{norm} we may write
\begin{eqnarray*}
&& \mathbf P \bigl(\bigl |\xi_{x,h}^1(X_t)\bigr|\ge\zeta
, \A^{\mathrm
{ll}}_{t,h} \bigr)
\\
&&\quad \le \mathbf P \bigl( \bigl\|\Xi_{x,h}^1(X_t)
\bigr\|_\infty\ge \zeta , \A^{\mathrm{ll}}_{t,h} \bigr)
\\
&&\quad \le \mathbf P \biggl( \frac{\llvert \int_0^t
\nu_0(t,s) \,\d B_s \rrvert }{\lambda_m(t)} \vee\frac{ \llvert \int_0^t \nu_1(t,s) \,\d B_s\rrvert  }{\lambda_M(t)}
\ge\zeta/\sqrt {2} , \A^{\mathrm{ll}}_{t,h} \biggr)
\\
&&\quad \le \mathbf P \biggl( \frac{\llvert \int_0^t
\nu_0(t,s) \,\d B_s\rrvert  }{\lambda_m(t)} \ge\zeta/\sqrt{2} , \A
^{\mathrm{ll}}_{t,h} \biggr) + \mathbf P \biggl( \frac{\llvert \int_0^t \nu_1(t,s) \,\d B_s\rrvert  }{\lambda_M(t)}
\ge\zeta/\sqrt {2} , \A^{\mathrm{ll}}_{t,h} \biggr).
\end{eqnarray*}
Since $\lambda_M(t)\ge\lambda_m(t) \ge \rho t^{1-H+\beta}$ on
the random set $\A^{\mathrm{ll}}_{t,h}$, it follows that:
%
%e52 #&#
\begin{eqnarray}
\label{diag} \mathbf P \bigl( \bigl|\xi_{x,h}^1(X_t)\bigr|
\ge\zeta , \A^{\mathrm
{ll}}_{t,h} \bigr) &\le& \mathbf P \biggl( \biggl
\llvert \int_0^t \nu_0(t,s)
\,\d B_s\biggr\rrvert \ge\frac{\rho\zeta}{\sqrt{2}} t^{1-H+\beta} \biggr)
\nonumber
\\[-8pt]
\\[-8pt]
\nonumber
&&{} + \mathbf P \biggl( \biggl\llvert \int
_0^t \nu_1(t,s) \,\d B_s
\biggr\rrvert \ge\frac{\rho\zeta}{\sqrt{2}} t^{1-H+\beta} \biggr) .
\end{eqnarray}
For $j=0,1$, $ |( \frac{X_s -x}h )^j N ( \frac{X_s -x}h ) | \le1$.
Then we may apply the exponential inequality \eqref{inegexp} and we obtain
%
%e53 #&#
\begin{equation}
\label{bound-xi} \mathbf P \bigl( \bigl|\xi_{x,h}^1(X_t)\bigr|
\ge\zeta , \A^{\mathrm
{ll}}_{t,h} \bigr) \le 4\exp \biggl( -
\frac{(1-H)\rho^2}{2\alpha_H^2} \zeta^2 t^{2\beta} \biggr) .
\end{equation}

Thanks to the decomposition \eqref{decomp} and the bounds \eqref
{bound-r} and \eqref{bound-xi}, we deduce that
\begin{eqnarray*}
&& \mathbf P \bigl( \bigl| \hat b^{\mathrm{ll}}_{t,h}(x)- b(x)\bigr|\ge
c_{\rho
,H} \Delta_{x,h} t^{1-H-\beta} + \zeta ,
\A^{\mathrm
{ll}}_{t,h} \bigr) \\ %
&&\quad \le\mathbf P
\bigl( \bigl|\xi _{x,h}^1(X_t)\bigr|+\bigl|r_{x,h}^\mathrm{loc}(X_t)\bigr|
\ge c_{\rho,H} \Delta _{x,h} t^{1-H-\beta} + \zeta ,
\A^{\mathrm{ll}}_{t,h} \bigr)
\\
&&\quad \le\mathbf P \bigl(\bigl |\xi_{x,h}^1(X_t)\bigr|
\ge\zeta , \A^{\mathrm{ll}}_{t,h} \bigr)
\\
&&\quad \le 4\exp \biggl( - \frac{(1-H)\rho
^2}{2\alpha_H^2}
\zeta^2 t^{2\beta} \biggr)
\end{eqnarray*}
and the proof of \eqref{poo} is now completed. %\hfill$\square$

%s5.1.2 #&#
\subsubsection{\texorpdfstring{Proof of Theorem~\textup{\protect\ref{th-dev-ll}(ii)}}
{Proof of Theorem 7(ii)}}
We follow the same arguments that the ones used in the proof of Theorem
\ref{as}. With $\beta<1-H$, we need to show that
%
%e54 #&#
\begin{equation}
\label{ii} \P \bigl(\Omega\setminus \A^{\mathrm{ll}}_{t,h} \bigr) = \P
\bigl( \lambda_m(t)\le\rho t^{1-H+\beta} \bigr) \mathop{
\longrightarrow}_{t\to\infty} 0 .
\end{equation}
Since $0\le N\le1$ and $z^2N(z)\le N(z)$ for any real $z$, $\lambda
_m(t)\ge\mathrm{v}_2(t)$ by \eqref{vp}.
Thus,
\begin{eqnarray*}
\P \bigl(\Omega\setminus \A^{\mathrm{ll}}_{t,h} \bigr) & \le& \P
\bigl( \mathrm{v}_2(t)\le\rho t^{1-H+\beta} \bigr)
\\
& \le&\P \biggl( \int_0^t
\alpha_H^2 (t-s)^{1-2H} \biggl( \frac{X_s -x}h
\biggr)^2 N \biggl( \frac{X_s -x}h \biggr) \,\d s\le\rho
t^{1-H+\beta} \biggr) .
\end{eqnarray*}
When $H>1/2$, $(t-s)^{1-2H}\ge t^{1-2H}$ and consequently
\begin{eqnarray*}
&&\biggl\{ \int_0^t \alpha_H^2
(t-s)^{1-2H} \biggl( \frac{X_s
-x}h \biggr)^2 N \biggl(
\frac{X_s -x}h \biggr) \,\d s\le \rho t^{1-H+\beta} \biggr\}
\\
&&\quad \subseteq \biggl\{ \frac{1}t\int_0^t
\biggl( \frac{X_s -x}h \biggr)^2 N \biggl( \frac{X_s -x}h
\biggr) \,\d s \le\frac{\rho}{\alpha_H^2} t^{H+\beta-1} \biggr\}.
\end{eqnarray*}
Since Hypothesis \ref{reg-inv-ll} holds, we may apply Proposition \ref
{ergo}. So \eqref{ii} is true and the result is proved for $H>1/2$.
When $H<1/2$, we use $(t-s)^{1-2H}\ge(t-\frac{t}2)^{1-2H} = (\frac
{t}2)^{1-2H}$ and we write
\begin{eqnarray*}
&&\biggl\{ \int_0^t \alpha_H^2
(t-s)^{1-2H} \biggl( \frac{X_s
-x}h \biggr)^2 N \biggl(
\frac{X_s -x}h \biggr)  \,\d s\le \rho t^{1-H+\beta} \biggr\}
\\
&&\quad \subseteq \biggl\{ \frac{1}{t/2} \int_0^{{t}/2}
\biggl( \frac{X_s -x}h \biggr)^2 N \biggl( \frac{X_s -x}h
\biggr) \,\d s \le\frac{4\rho}{\alpha
_H^2} \biggl( \frac{t}2
\biggr)^{H+\beta-1} \biggr\}.
\end{eqnarray*}
The concluding arguments are unchanged.
%
%s5.2 #&#
\subsection{\texorpdfstring{Proof of Theorem \protect\ref{didill}}
{Proof of Theorem 8}}
We treat separately each term of the decomposition \eqref{decompn}.
Repeating the arguments that led us to~\eqref{diag} yields that on
$\mathcal A^{\mathrm{ll}}_{t_n,h}$
\begin{eqnarray*}
\mathbf P \bigl( \bigl|\xi_{x,h}(X_{t_n})\bigr|\ge\zeta ,
\A^{\mathrm
{ll}}_{t_n,h} \bigr) &\le& \mathbf P \biggl( \biggl\llvert \int
_0^{t_n}\mu^0(t_n,s)
\,\d B_s\biggr\rrvert \ge\frac{\rho\zeta}{\sqrt{2}} t_n^{1-H+\beta}
\biggr)
\\
&&{} + \mathbf P \biggl( \biggl\llvert \int
_0^{t_n} \mu ^1(t_n,s)
\,\d B_s\biggr\rrvert \ge\frac{\rho\zeta}{\sqrt{2}} t_n^{1-H+\beta}
\biggr) .
\end{eqnarray*}
We follow the arguments that led us to \eqref{arghh}.
We fix $t_n$ and for $j=0,1$, we consider the martingales
$Z^{n,j}:=(Z^{n,j}_r)_{0\le r\le t_n}$ defined by
\[
Z^{n,j}_r = \int_0^r
\mu^j(t_n,s) \,\d B_s .
\]
Since $ |( \frac{X_s -x}h )^j N ( \frac{X_s -x}h ) | \le1$, the
quadratic variations of the martingales $Z^{n,j}$ satisfy
\[
\bigl\langle Z^{n,j} \bigr\rangle_r  \le
\frac{2  \alpha_H^2}{1-H} t_n^{2-2H}.
\]
Then by Lemma \ref{exp} we obtain that
%
%e55 #&#
\begin{equation}
\label{arghhn} \mathbf P \bigl( \bigl|\xi_{x,h}(X_{t_n})\bigr|\ge\zeta
, \A^{\mathrm
{ll}}_{t_n,h} \bigr) \le4\exp \biggl( -\frac{\rho^2
(1-H)  \zeta^2}{8  \alpha_H^2}
t_n^{2\beta} \biggr) .
\end{equation}
Now we deal with $r_{x,h}^{\mathrm{traj}}(X_{t_n})$ and
$r_{x,h}^{\mathrm{loc}}(X_{t_n})$.
For $j=0,1$ we have a discrete version of~\eqref{nunu}
\[
\biggl\llvert \int_0^t \varpi^j_{\mathrm{loc}}(t_n,s)
\,\d s\biggr\rrvert \le\Delta _{x,h} \v^n_0(t_n)
.
\]
Arguing as in the proof of \eqref{bound-r},
%
%e56 #&#
\begin{equation}
\label{bound-rn} \bigl|r_{x,h}^{\mathrm{loc}}(X_{t_n}) \bigr| \le
c_{\rho,H} \Delta_{x,h} t_n^{1-H-\beta}
\end{equation}
on the random set $\mathcal A^{\mathrm{ll}}_{t_n,h}$ (we recall that
$c_{\rho,H}=\sqrt{2} c_H^2 /(\rho\lambda_H )$).
We use \eqref{nx} and we obtain similarly,
\[
\biggl\llvert \int_0^t \varpi^j_{\mathrm{traj}}(t_n,s)
\,\d s\biggr\rrvert \le \bigl( c_{x,h,L} \epsilon_n + L
\epsilon_n^\h \bigl\| B^H\bigr \| _{0,t_n,\mathfrak{h}}
\bigr) \v^n_0(t_n) ,
\]
with $c_{x,h,L}=L  \{ L  h+c_b(1+|x|^\mm) \}$. Finally on
the set $\mathcal A^{\mathrm{ll}}_{t_n,h}$ it holds
%
%e57 #&#
\begin{equation}
\label{bound-rdeuxn} \bigl|r_{x,h}^{\mathrm{traj}}(X_{t_n}) \bigr| \le
\bigl( c_1 \epsilon_n + c_2
\epsilon_n^\h \bigl\| B^H \bigr\|_{0,t_n,\mathfrak{h}}
\bigr) t_n^{1-H-\beta}
\end{equation}
with $c_1=c_{\rho,H}  c_{x,h,L}$ and $c_2=c_{\rho,H}  L$.
Like in the proof of \eqref{eqdidi} (see Section \ref{ji}), we
combine the inequalities \eqref{arghhn}, \eqref{bound-rn} and
\eqref{bound-rdeuxn} and we deduce:
\begin{eqnarray}
\label{yun} &&\mathbf P \bigl( \bigl| \hat b^{\mathrm{ll}}_{t_n,h}(x)- b(x)\bigr|
\ge c_{\rho,H} \Delta_{x,h} t_n^{1-H-\beta}+c_1
\epsilon_n t_n^{1-H-\beta} + \zeta ,
\A^{\mathrm{\sc{nw}}}_{t_n,h} \bigr)
\nonumber
\\[-8pt]
\\[-8pt]
\nonumber
& &\quad \le4\exp \biggl( -
\frac{\rho^2  (1-H)  \zeta^2}{16
\alpha_H^2} t_n^{2\beta} \biggr)+ \mathbf P \biggl( \bigl\|
B^H \bigr\| _{0,t_n,\mathfrak{h}}\ge\frac{\zeta}{2 c_2 \epsilon_n^\h
t_n^{1-H-\beta}} \biggr) .
\end{eqnarray}\eject\noindent
By Lemma \ref{fernique}, we have
\begin{eqnarray*}
&&\mathbf P \biggl( \bigl\| B^H \bigr\|_{0,t_n,\mathfrak{h}}\ge\frac
{\zeta}{2 c_2 \epsilon_n^\h t_n^{1-H-\beta}}
\biggr)\\
&&\quad  \le c_{H,\mathfrak{h}} \bigl(1 +t_n^{H-\mathfrak{h}} \bigr)\exp \biggl( -\frac{\zeta}{2 c_2}
t_n^{\h(\gamma-1)+H-1+\beta} \biggl( 1 - \frac{256 H^2  c_2}{\zeta \mathfrak{h}^2}
t_n^{2(H-\mathfrak{h})} t_n^{-\h(\gamma-1)-H+1-\beta} \biggr) \biggr)
\\
&&\quad \le c_{H,\mathfrak{h}} \bigl(1 +t_n^{H-\mathfrak{h}} \bigr)\exp \biggl( -\frac{\zeta}{2 c_2}
t_n^{\h(\gamma-1)+H-1+\beta} \biggl( 1 - \frac{256 H^2  c_2}{\zeta \mathfrak{h}^2}
t_n^{-\h
(\gamma+1)+H+1-\beta} \biggr) \biggr). %&
% \le c_{H,\mathfrak{h}}  \big(1+t_n^{H-\mathfrak{h}}\big)  \exp
% \left( 1 -
% C   t_n^{2H-\mathfrak{h}-\gamma\h}\right)\right)
\end{eqnarray*}
With $\beta<1-\gamma H$ one may choose $\h$ such that
\[
\max \biggl( \frac{1-H-\beta}{\gamma-1} , \frac
{1+H-\beta}{\gamma+1} \biggr) <\h<H .
\]
When $t_n$ is large enough, we obtain
\begin{eqnarray*}
\mathbf P \biggl( \bigl\| B^H\bigr \|_{0,t_n,\mathfrak{h}}\ge\frac
{\zeta}{2 c_2 \epsilon_n^\h t_n^{1-H-\beta}}
\biggr) \le 2 c_{H,\mathfrak{h}} t_n^{H-\mathfrak{h}} \exp \biggl( -
\frac{\zeta}{4 c_2} t_n^{\h(\gamma-1)+H-1+\beta} \biggr) .
\end{eqnarray*}
We report the above estimation in \eqref{yun} and the proof is
complete when we set $\mathfrak{u}_1 = H-\h$ and
$\mathfrak{u}_1 = \h(\gamma-1)+H-1+\beta$.
\begin{appendix}

%s6 #&#
\section{A Fernique's type lemma}\label{app1}
The exponential moments of the H\"{o}lder norm of the trajectories of a
fBm are classical results
from the theory of Gaussian processes (see \cite{fer}, e.g.).
Nevertheless we are interested in the large time behaviour
of this moment. So we prove in this Appendix the following Fernique's
type lemma in which we give precision on the time dependence
of the estimation.
%
%le10 #&#
\begin{lemme}\label{fernique}
Let $T>0$, $0< \mathfrak{h} <H <1$. We denote
\[
\bigl\| B^H \bigr\|_{0,T,\mathfrak{h}}=\sup_{0\le s,t\le T}
\frac
{|B_t^H-B_s^H|}{|t-s|^\mathfrak{h}} .
\]
Then for any $T>T_0=(\mathfrak{h}/(8H))^{1/(H-\mathfrak{h})}$ there
exists a constant $c_{H,\mathfrak{h}}$ such that
%
%e58 #&#
\renewcommand{\theequation}{\arabic{equation}}
\setcounter{equation}{58}
\begin{equation}
\label{esti-exp} \E \bigl[ \exp \bigl( \bigl\| B^H \bigr\|_{0,T,\mathfrak{h}}
\bigr) \bigr] \leq c_{H,\mathfrak{h}} \bigl(1+T^{H-\mathfrak{h}} \bigr) \exp \biggl(
\frac{128 H^2}{\mathfrak{h}^2} T^{2(H-\mathfrak
{h})} \biggr) .
\end{equation}
\end{lemme}
The explicit form of $c_{H,\mathfrak{h}}$ is given in \eqref{ccc}.
\begin{pf*}{Proof of Lemma \ref{fernique}}
We denote $[z]$ the integer part of a non-negative real $z$. First, we
prove that
%
%e59 #&#
\renewcommand{\theequation}{\arabic{equation}}
\setcounter{equation}{59}
\begin{equation}
\label{modul} \bigl|B^{H}_t - B^{H}_s \bigr|
\leq \xi_{H,\mathfrak{h},T} |t-s|^{\mathfrak{h}},
\end{equation}
where $\xi_{H,\mathfrak{h},T}$ is a positive random variable such
that for any $p>p_0:=[2/(H-\mathfrak{h})]$
%
%e60 #&#
\begin{equation}
\label{modul1}  \E \bigl( \xi_{H,\mathfrak{h},T}^{p} \bigr)\leq \biggl(
16 \frac{H}{\mathfrak{h}} \biggr)^p T^{p(H-\mathfrak{h})} (p-1)!! .
\end{equation}
The double factorial of a positive integer $p$ is defined by
\[
\cases{ %
\displaystyle (2k-1)!! = \prod
_{i=1}^k (2i-1)=\frac{(2k)!}{2^k  k!} ;
\vspace*{2pt}\cr
\displaystyle (2k)!!  = \prod_{i=1}^k
(2i)=2^k k! .}
\]
We remark that when $p\le p_0$, we obtain easily that
%
%e61 #&#
\begin{equation}
\label{modul1bis}  \E \bigl( \xi_{H,\mathfrak{h},T}^{p} \bigr)\leq
\biggl( 16 \frac{H}{\mathfrak{h}} \biggr)^p T^{p(H-\mathfrak{h})}
p_0!! .
\end{equation}
In order to prove \eqref{modul} and \eqref{modul1}, we proceed as
follows. With $\psi(u)= u^{2/(H-\mathfrak{h})}$ and $p(u)=u^{H}$ in
Lemma 1.1 of \cite{grr}, the Garsia--Rodemich--Rumsey inequality reads
\[
\bigl|B^{H}_t-B^{H}_s\bigr| \leq8 \int
_0^{|t-s|} \biggl( \frac{4 \mathfrak
{D}}{u^2}
\biggr)^{(H-\mathfrak{h})/2} H u^{H-1} \,\d u,
\]
where the random variable $\mathfrak{D}$ is
\[
\mathfrak{D} = \int_0^T\int
_0^T \frac
{|B^{H}_t-B^{H}_s|^{2/(H-\mathfrak{h})}}{|t-s|^{2H/(H-\mathfrak{h})}} \,\d t\,\d s .
\]
We have
\begin{eqnarray*}
\bigl|B^{H}_t-B^{H}_s\bigr| & \leq& 8 (4
\mathfrak{D})^{(H-\mathfrak{h})/2} \int_0^{|t-s|}H
u^{\mathfrak{h}-1} \,\d u
\\
& \leq& 8 \frac{H}{\mathfrak{h}} (4\mathfrak {D})^{(H-\mathfrak{h})/2}
|t-s|^{\mathfrak{h}} .
\end{eqnarray*}
We denote $\xi_{H,\mathfrak{h},T} = 8   \frac{H}{\mathfrak
{h}}   (4\mathfrak{D})^{(H-\mathfrak{h})/2}$. By Jensen's
inequality, for $p\ge2/(H-\mathfrak{h})$ it holds
\begin{eqnarray*}
\E \bigl(\xi_{H,\mathfrak{h},T}^{p} \bigr) & \leq& \biggl( 8
4^{(H-\mathfrak{h})/2} \frac
{H}{\mathfrak{h}} \biggr)^p \E \biggl( \int
_0^T\int_0^T
\frac
{|B^{H}_t-B^{H}_s|^{2/(H-\mathfrak{h})}}{|t-s|^{2H/(H-\mathfrak{h})}} \,\d t\,\d s \biggr)^{p(H-\mathfrak{h})/2}
\\
& \leq &\biggl( 16 \frac{H}{\mathfrak{h}} \biggr)^p
T^{p(H-\mathfrak{h})} \int_0^T\int
_0^T \frac{\E (
|B^{H}_t-B^{H}_s|^{2p} )}{|t-s|^{2pH}} \frac{\mathrm{d}t\,\mathrm{d}s}{T^2}
\\
& \leq& \biggl( 16 \frac{H}{\mathfrak{h}} \biggr)^p T^{p(H-\mathfrak{h})}
\E\bigl( |Z|^p\bigr) ,
\end{eqnarray*}
where $Z$ is a Gaussian random variable with zero mean and unit variance.
Since
\[
\E\bigl(|Z|^p\bigr) = \cases{ %
\sqrt{2/\pi} (p-1)!!, & \quad$\mbox{when $p$ is odd;}$
\vspace*{2pt}\cr
(p-1)!!, &\quad $\mbox{when $p$ is even,}$}
\]
we deduce \eqref{modul} and \eqref{modul1}.

What remains to be shown
can be tediously deduced from Theorem 1.3.2 in \cite{fer}. We can also
make the following direct computations. Using \eqref{modul}, \eqref
{modul1} and \eqref{modul1bis}, we have
\begin{eqnarray*}
\E \bigl( \exp \bigl( \bigl\| B^H \bigr\|_{0,T,\mathfrak{h}} \bigr) \bigr) &
\leq&
\E \bigl( \exp ( \xi_{H,\mathfrak{h},T} ) \bigr)
\\
& \leq& \sum_{p=0}^\infty\frac{\E(\xi_{H,\mathfrak
{h},T}^{p})}{p!}
\\
& \leq& p_0!! \sum_{p=0}^{p_0}
\frac{c_{H,\mathfrak{h},T}^p}{p!} + \sum_{p=p_0+1}^{\infty}
c_{H,\mathfrak{h},T}^p\frac{(p-1)!!}{p!}
\\
& \le& p_0!! \exp(c_{H,\mathfrak{h},T}) + \sum
_{p=0}^\infty c_{H,\mathfrak{h},T}^p
\frac{(p-1)!!}{p!}
\end{eqnarray*}
where we have denoted $c_{H,\mathfrak{h},T}= 16 (H/\mathfrak
{h})T^{H-\mathfrak{h}}$.
We notice that
\[
(p-1)!!/p! = \cases{ %
 1/\bigl(2^k k
\bigr)!, & \quad$\mbox{when $p=2k$;}$
\vspace*{2pt}\cr
1/(2k+1)!!, &\quad $\mbox{when $p=2k+1$.}$}
\]
Since $(2k+1)!!\ge\prod_{i=1}^k 2k =2^kk!$, we obtain
%
%e62 #&#
\begin{eqnarray}
\label{lm} \E \bigl( \exp \bigl( \bigl\| B^H \bigr\|_{0,T,\mathfrak{h}} \bigr)
\bigr) & \le& p_0!! \exp(c_{H,\mathfrak{h},T}) + \sum
_{k=0}^\infty\frac
{c_{H,\mathfrak{h},T}^{2k}}{2^kk!} + \sum
_{k=0}^\infty\frac
{c_{H,\mathfrak{h},T}^{2k+1}}{2^kk!}
\nonumber
\\
& \le &p_0!! \exp(c_{H,\mathfrak{h},T}) + \sum
_{k=0}^\infty \biggl(\frac{c_{H,\mathfrak{h},T}^2}2
\biggr)^k \frac{1}{k!} + c_{H,\mathfrak{h},T}\sum
_{k=0}^\infty \biggl( \frac
{c_{H,\mathfrak{h},T}^2}2
\biggr)^k \frac{1}{k!}\qquad
\\
& \le& p_0!! \exp(c_{H,\mathfrak{h},T})+ ( 1+c_{H,\mathfrak
{h},T} ) \exp
\bigl(c_{H,\mathfrak{h},T}^2/2\bigr) .\nonumber
\end{eqnarray}
With
%
%e63 #&#
\begin{equation}
\label{ccc} c_{H,\mathfrak{h}}  = \biggl[\frac{2}{H-\mathfrak
{h}} \biggr] +
\frac{16 H}{\mathfrak{h}} ,
\end{equation}
the lemma is proved because when $T^{H-\mathfrak{h}}\ge\mathfrak
{h}/(8H)$, $c_{H,\mathfrak{h},T}\le c_{H,\mathfrak{h},T}^2/2$.\vadjust{\goodbreak}
\end{pf*}
We remark that if $T$ do not satisfy the condition $T\ge\mathfrak
{h}/(8H)$, then one may replace \eqref{esti-exp} by~\eqref{lm}.
%
%re7 #&#
\begin{remi}\label{rem-momo}
Thanks to \eqref{modul} and \eqref{modul1}, we have obtained the
following estimation for the moments of the H\"{o}lder norm of the
trajectories of the fBm:
%
%e64 #&#
\renewcommand{\theequation}{\arabic{equation}}
\setcounter{equation}{64}
\begin{equation}
\label{momo} {\mathbf E} \bigl( \bigl\| B^H \bigr\|_{0,T,\mathfrak{h}}^p
\bigr)  \le \biggl( 16 \frac{H}{\mathfrak{h}} \biggr)^p
T^{p(H-\mathfrak
{h})} (p-1)!!
\end{equation}
for any $p>p_0:=[2/(H-\mathfrak{h})]$.
\end{remi}
%
%s7 #&#
\section{\texorpdfstring{Proof of Proposition \protect\ref{ergo}}
{Proof of Proposition 1}}\label{ergo-proof}

\textit{Step \textup{1:} Proof of \textup{\eqref{ergo-lim1}}}.
We use the inequality
%
%e65 #&#
\renewcommand{\theequation}{\arabic{equation}}
\setcounter{equation}{65}
\begin{eqnarray}
\label{uiui} \biggl\llvert \frac{1}T \int_0^T
\ffi(X_t)\,\d t - \E \bigl( \ffi(\bar X) \bigr)\biggr\rrvert & \le&
\frac{1}T \int_0^T \bigl\llvert
\ffi(X_t)-\ffi\bigl(\bar X(\theta_t)\bigr) \bigr\rrvert \,\d t
\nonumber
\\[-8pt]
\\[-8pt]
\nonumber
&&{} + \biggl\llvert \frac{1}T \int_0^T
\ffi\bigl(\bar X(\theta_t)\bigr)\,\d t - \E \bigl( \ffi(\bar X) \bigr)
\biggr\rrvert . % & \le\frac{\|\ffi'\|_\infty}T \int_0^T \left|X_t-\bar X(\theta_t)
\end{eqnarray}
Since $\ffi$ has polynomial growth and $\bar X$ has moment of any
order, $\ffi(\bar X)$ is an integrable random variable and
\eqref{err} implies that the second term in the right-hand side of
\eqref{uiui} tends to $0$ almost-surely.
Now we treat the first one. The inequalities
\begin{eqnarray*}
\bigl\llvert \ffi(X_t)-\ffi\bigl(\bar X(\theta_t)\bigr)
\bigr\rrvert & \le& c_\ffi \bigl( 1+|X_t|^\ppp+ \bigl|
\bar X(\theta_t)\bigr|^\ppp \bigr) \bigl|X_t-\bar X(
\theta _t)\bigr |
\\
& \le& c_{\ffi,\ppp} \bigl( 1+\bigl|X_t-\bar X(\theta_t)\bigr|^\ppp+
2\bigl|\bar X(\theta_t)\bigr|^\ppp \bigr) \bigl|X_t-\bar X(
\theta_t) \bigr|
\end{eqnarray*}
imply that
%
%e66 #&#
\begin{eqnarray}
\label{popo} \frac{1}T \int_0^T \bigl
\llvert \ffi(X_t)-\ffi\bigl(\bar X(\theta_t)\bigr) \bigr
\rrvert \,\d t & \le& \frac{c_{\ffi,\ppp}}T \int_0^T
\bigl|X_t-\bar X(\theta_t) \bigr|^{\ppp+1}\,\d t
\nonumber
\\[-8pt]
\\[-8pt]
\nonumber
&&{} + \frac{c_{\ffi,\ppp}}T \int_0^T
\bigl|\bar X(\theta _t)\bigr|^\ppp\bigl|X_t-\bar X(
\theta_t) \bigr|\,\d t .
\end{eqnarray}
By \eqref{err2}, $|X_t-\bar X(\theta_t)|^{\ppp+1}$ tends to $0$
almost-surely and an integral version of the Toeplitz lemma
implies that the first term in the right-hand side of \eqref{popo}
tends to $0$ almost-surely. For the second one, by the
Cauchy--Schwarz inequality,
\[
\biggl( \frac{1}T \int_0^T \bigl|\bar X(
\theta_t)\bigr|^\ppp\bigl|X_t-\bar X(
\theta_t) \bigr|\,\d t \biggr)^2 \le \biggl( \frac{1}T \int
_0^T \bigl|\bar X(\theta_t)\bigr|^{2\ppp}
\,\d t \biggr) \times \biggl( \frac{1}T \int_0^T
\bigl|X_t-\bar X(\theta_t) \bigr|^2\,\d t \biggr)
\]
and thus it tends to $0$ by the same arguments that we employed before.
The proof of \eqref{ergo-lim1} is complete.

\textit{Step \textup{2:} Proof of \textup{\eqref{ergo-lim2}}}.
First, we write
\[
\Biggl\llvert \frac{1}{t_n} \int_0^{t_n}
\Biggl\{ \sum_{k=0}^{n-1} \ffi
(X_{t_k}) \1_{[t_k{,}t_{k+1})}(s) \Biggr\} \,\d s - \E \bigl( \ffi(\bar X)
\bigr) \Biggr\rrvert \le I_n^1+I_n^2
\]
with
\begin{eqnarray*}
I_n^1 & =& \frac{1}{t_n} \int_0^{t_n}
\Biggl\{ \sum_{k=0}^{n-1} \bigl\llvert
\ffi(X_{t_k}) -\ffi(X_s)\bigr\rrvert \1_{[t_k{,}t_{k+1})}(s)
\Biggr\} \,\d s  \quad\mbox{and}
\\
I_n^2 & =& \biggl\llvert \frac{1}{t_n} \int
_0^{t_n} \ffi(X_s)\,\d s - \E \bigl( \ffi(
\bar X) \bigr)\biggr\rrvert .
\end{eqnarray*}
By \eqref{ergo-lim1}, $\lim_{n\to\infty}I_n^2=0$ almost-surely. We
estimate $I_n^1$ as follows. First, of all we write:
%
%e67 #&#
\begin{eqnarray}
\label{in1} I_n^1 & \le &\frac{c_\ffi}{t_n} \int
_0^{t_n} \Biggl\{ \sum
_{k=0}^{n-1} \bigl( 1+\bigl|X_{t_k}|^\ppp+
|X_s\bigr|^\ppp \bigr) \llvert X_{t_k}
-X_s\rrvert \1_{[t_k{,}t_{k+1})}(s) \Biggr\} \,\d s
\nonumber
\\[-8pt]
\\[-8pt]
\nonumber
& \le &\frac{c_\ffi}{t_n} \Bigl( 1+\sup_{0\le u\le t_n}|X_u|^\ppp
\Bigr) \int_0^{t_n} \Biggl\{ \sum
_{k=0}^{n-1} \llvert X_{t_k}
-X_s\rrvert \1_{[t_k{,}t_{k+1})}(s) \Biggr\} \,\d s . %\\ %
%& \le\|\ffi'\|_\infty\ \epsilon_n^\h\ \| X \|_{0,t_n,\mathfrak{h}}\ ,
\end{eqnarray}
%
%where we have denoted for $0<\h<H$:
%$$\| X \|_{0,t_n,\mathfrak{h}}=\sup_{0\le r,s\le t_n}
Since $b$ satisfies the polynomial growth condition, it holds for
$t_k\le s\le t_{k+1}$ that
%
%e68 #&#
\begin{eqnarray}
\label{zaza0} |X_s-X_{t_k}| & \le&\int_{t_k}^s
\bigl|b(X_u)\bigr|\,\d u + \bigl|B_s^H-B_{t_k}^H\bigr|
\nonumber
\\
& \le& c_b \int_{t_k}^s \bigl( 1 +
|X_u|^\mm \bigr) \,\d u + \bigl\| B^H \bigr\|
_{0,t_n,\mathfrak{h}} |s-{t_k}|^\h
\\
& \le & c_b \Bigl( 1+ \Bigl(\sup_{0\le u\le t_n}
|X_u| \Bigr)^\mm \Bigr) \epsilon_n + \bigl\|
B^H \bigr\|_{0,t_n,\mathfrak{h}} \epsilon _n^\h .\nonumber
\end{eqnarray}
%
%and consequently
% \| X \|_{0,t_n,\mathfrak{h}} & \le C_b  t_n^{1-\h} \big( 1+ \sup_{0
Under the one-sided dissipative Lipschitz condition, Proposition 1 in
\cite{kn} establishes that
\begin{eqnarray*}
\sup_{0\le u\le t_n} |X_u| & \le& c_b \Bigl(
1+|x_0| + \Bigl(\sup_{0\le u \le t_n} \bigl|B_u^H\bigr|
\Bigr)^{\mathfrak{m}} \Bigr)
\\
& \le& c_b \bigl( 1+|x_0| + \bigl\| B^H
\bigr\|^\mm_{0,t_n,\mathfrak{h}} t_n^{\mm\h} \bigr) ,
\end{eqnarray*}
with $0<\h<H$ that will be fixed later. We report the above inequality
in \eqref{zaza0}:
%
%e69 #&#
\begin{equation}
\label{zaza2} |X_s-X_{t_k}| \le c_{b,x_0}
\epsilon_n + c_b \bigl\| B^H
\bigr\|^{\mm
^2}_{0,t_n,\mathfrak{h}} t_n^{\mm^2\h}
\epsilon_n + \bigl\| B^H \bigr\| _{0,t_n,\mathfrak{h}}
\epsilon_n^\h,\qquad t_k\le s\le t_{k+1}
.
\end{equation}
Using \eqref{zaza2} in \eqref{in1}, we deduce that there exist a
constant $C$ that depends on $b$, $\ffi$ and $x_0$ such that
\begin{eqnarray*}
I_n^1 & \le& C \bigl( 1+ \bigl\| B^H
\bigr\|^{\ppp}_{0,t_n,\mathfrak{h}} t_n^{\ppp\h} \bigr)\times
\bigl( \epsilon_n + \bigl\| B^H\bigr \|^{\mm
^2}_{0,t_n,\mathfrak{h}}
t_n^{\mm^2\h} \epsilon_n + \bigl\| B^H\bigr \|
_{0,t_n,\mathfrak{h}} \epsilon_n^\h \bigr)
\\
& \le& C \bigl( \epsilon_n +\bigl\| B^H \bigr\|_{0,t_n,\mathfrak{h}}
\epsilon_n^\h+ \bigl\| B^H \bigr\|_{0,t_n,\mathfrak{h}}^{\ppp+1}
t_n^{\ppp
\h} \epsilon_n^\h+ \bigl\|
B^H \bigr\|^{\mm^2+\ppp}_{0,t_n,\mathfrak{h}} t_n^{(\mm^2+\ppp)\h}
\epsilon_n \bigr) . %\\ %
%& \le\|\ffi'\|_\infty\ \epsilon_n^\h\ \| X \|_{0,t_n,\mathfrak{h}}\ ,
\end{eqnarray*}
The almost-sure convergence of $I_n^1$ to $0$ will follow from a
Borel--Cantelli argument. Indeed for any $\eta>0$ and for an integer
$q$ that will be chosen later, it holds that
\begin{eqnarray*}
{\mathbf P}\bigl( \bigl|I_n^1\bigr|\ge\eta\bigr) &\le&
\frac{C}{\eta} \bigl( \epsilon_n^q +
\epsilon_n^{\h q} {\mathbf E} \bigl(\bigl\| B^H
\bigr\|^q_{0,t_n,\mathfrak
{h}} \bigr) + t_n^{\h q\ppp}
\epsilon_n^{q\h} \E \bigl( \bigl\| B^H
\bigr\|^{q(\ppp
+1)}_{0,t_n,\mathfrak{h}} \bigr)\\
&&{}+ t_n^{\h q(\mm^2+\ppp)} \epsilon_n^q \E
\bigl( \bigl\| B^H \bigr\| ^{q(\mm^2+\ppp)}_{0,t_n,\mathfrak{h}} \bigr) \bigr)
\end{eqnarray*}
and by \eqref{momo} (see Remark \ref{rem-momo} in Appendix \ref
{app1}) we obtain
\begin{eqnarray*}
{\mathbf P}\bigl( \bigl|I_n^1\bigr|\ge\eta\bigr) & \le&
\frac{C}{\eta} \bigl( \epsilon_n^q +
\epsilon_n^{\h q} t_n^{q(H-\h)} +
t_n^{\h q\ppp} \epsilon_n^{q\h}
t_n^{q(\ppp+1)(H-\h)} + t_n^{\h q(\mm^2+\ppp)}
\epsilon_n^q t_n^{q(\mm^2+\ppp)(H-\h
)} \bigr)
\\
& \le& \frac{C}{\eta} \bigl( \epsilon_n^q +
\epsilon_n^{\h q} t_n^{q(H-\h)} +
\epsilon_n^{q\h} t_n^{q(\ppp+1)H -q\h} +
\epsilon_n^q t_n^{q(\mm^2+\ppp)H} \bigr) .
\end{eqnarray*}
Since $t_n^\gamma=n$ and $\epsilon_n=n^{-(\gamma-1)/\gamma}$, we
have $
\sum_{n\ge1} {\mathbf P}( |I_n^1|\ge\eta) \le\frac{C}{\eta}
( S_1+S_2+S_3+S_4 )
$
% \sum_{n\ge1} \p( |I_n^1|\ge\eta) & \le\frac{C}{\eta}\left(
%S_1+S_2+S_3+S_4\right)
with
\begin{eqnarray*}
S_1 & = &\sum_{n\ge1} \frac{1}{n^{q(\gamma-1)/\gamma}},
\\
S_2 & =& \sum_{n\ge1} \frac{1}{n^{q(\h\gamma- H)/\gamma}},
\\
S_3 & = &\sum_{n\ge1}
\frac{1}{n^{q(\h\gamma-(\ppp+1) H)/\gamma}}\quad \mbox{and}
\\
S_4 & =& \sum_{n\ge1}
\frac{1}{n^{q(\gamma-1-(\mm^2+\ppp)H)/\gamma
}} .
\end{eqnarray*}
It is supposed that $\gamma>1+(\mm^2 +\ppp)H$. We choose $\h$ close
to $H$ in such a way that $\h\gamma-H>0$.
Moreover, since $\gamma>\ppp+1$, one may choose $\h$ such that it
satisfies additionally $\gamma H> \gamma\h>(\ppp+1)H$.
Now it is clear that we may find an integer $q$ in such a way that the
three above sums converge. The Borel--Cantelli lemma yields that $I_n^1$
converges to $0$ almost-surely.
\end{appendix}
%
%
% imsref loaded by akundreckaite, 2013-03-25 08:02:11

%


%suskaldyti doi

\printhistory

\end{document}